\documentclass[11pt]{amsart}
\usepackage{amsmath,amscd,amssymb}
\usepackage{hyperref}
\newtheorem{theorem}{Theorem}[section]

\newtheorem{lemma}[theorem]{Lemma}
\newtheorem{proposition}[theorem]{Proposition}
\newtheorem{remark}[theorem]{Remark}
\numberwithin{equation}{section}
\newtheorem{definition}[theorem]{Definition}

\newtheorem{corollary}[theorem]{Corollary}

\input xy
\xyoption{all}

\begin{document}

\title[Brill-Noether loci on surface]
{ Geometry of certain Brill-Noether locus on a very general sextic surface and Ulrich bundles}

\author[D. Bhattacharya]{Debojyoti Bhattacharya }
\address{Indian Institute of Science Education and Research,Thiruvananthapuram,
Maruthamala PO, Vithura,
Thiruvananthapuram - 695551, Kerala, India}
\email{debojyotibhattacharya15@iisertvm.ac.in}








\begin{abstract}
Let $X \subset \mathbb P^3$ be a very general sextic surface over complex numbers.  In this paper we study certain Brill-Noether problems for moduli of rank $2$ stable bundles on $X$ and its relation with rank $2$ weakly Ulrich and Ulrich bundles. In particular, we  show the  non-emptiness of certain Brill-Noether loci and  using the geometry of the moduli and the notion of the Petri map on higher dimensional varieties,  we prove  the existence of components of expected dimension. We also give sufficient conditions for the existence of rank $2$ weakly Ulrich bundles $\mathcal E$ on $X$ with $c_1(\mathcal E) =5H$ and $c_2 \geq 91$ and  partially address the question of whether these conditions really hold. We then study the possible  implication of the existence of an weakly Ulrich bundle in terms of non-emptiness of Brill-Noether loci. Finally, using the existence of rank $2$ Ulrich bundles on $X$ we obtain some more non-empty Brill-Noether loci and investigate the possibility of existence of higher rank simple Ulrich bundles on $X$.  
\end{abstract}





\maketitle\section{Introduction and statement of the main results}
Let $X$ be a smooth, projective, irreducible  variety of dimension $n$ over $\mathbb C$  and $H$ be an ample divisor on $X$. Let $\mathcal M_{H} : = \mathcal M_{X,H}(r;c_1,...,c_s)$ be the moduli space of rank $r$, $H$-stable (in the sense of Mumford and Takemoto) vector bundles $\mathcal E$ on $X$ with fixed Chern classes $c_i(\mathcal E) = c_i \in H^{2i}(X,\mathbb Z)$ for $i =1,...,s : =\text{min}\{r,n\}$. A Brill-Noether locus $\mathcal W^{k}_{r,H}$ is a subvariety of $\mathcal M_H$ whose support correspond to points $\mathcal E \in \mathcal M_{H}$ having atleast $k+1$ independent global sections (i.e $h^0(X, \mathcal E) \geq k+1$). Brill-Noether theory describes the geometry of such varieties. Classical Brill-Noether theory has its roots in the $19$th century work of Brill and Noether and in modern language of vector bundles it's essentially the study of line bundles on algebraic curves (see \cite{ACGH}). For smooth irreducible projective curve $X$ over $\mathbb C$, the Brill-Noether locus of the moduli space of degree $d$ line bundles $L$ on $X$(i.e $\text{Pic}^{d}(X)$) is extensively studied for a very long period of time. In fact, if $X$ is a Petri curve (then $X$ is general in the sense of Brill-Noether theory), then answers to the  problems  like non-emptiness, irreducibility, connectedness, singular locus of the Brill-Noether loci are completely known (see \cite{ACGH}). As a systematic generalization of line bundles on algebraic curves, Brill-Noether problems on the moduli space of rank $r \geq 2$ stable vector bundles on algebraic curves for arbitrarily large $r$ are investigated by Newstead, Teixidor and other authors. For a detailed account of results in this case see \cite{alternativepetri} and references mentioned there. Moving to higher dimensional varieties,  G{\"o}ttsche, Hirschowitz in \cite{LG1} and M. He in \cite{HeMin} studied the Brill-Noether loci of moduli of stable bundles on $\mathbb P^2$, Yoshioka(\cite{Y1},\cite{Y2}), Markman(\cite{Markman}), Leyenson(\cite{Leyenson})   studied Brill-Noether problems on $K3$ surfaces,  Nakashima studied Brill-Noether problems for stable sheaves on projective varieties of higher dimension (\cite{Nakashima}). However, despite all this effort very little is known about the basic questions like whether the Brill-Noether locus is non-empty, connected, irreducible or has components of expected dimension and a lot of new properties  appear in this higher dimensional context making this new theory a growing field of current research. \\

In \cite{LC1}, as a  generalization of classical Brill-Noether theory it's shown that natural scheme structure can be given on the Brill-noether locus $\mathcal W^{k}_{r,H}$ on a smooth projective variety of dimension $n\geq 2$,  under the assumption that all higher order cohomology vanishes (i.e for all $\mathcal E \in \mathcal M_H, H^i(\mathcal E) = 0, \forall i \geq 2$) and several concrete examples like Brill-Noether problems for Mathematical Instanton bundles on $\mathbb P^3$ (cf.\cite{LC1}), stable vector bundles on Hirzebruch surfaces (cf.\cite{LC2}) are studied in recent times and several examples were given where expected dimension is not same as the actual dimension. We mention that in all the examples illustrated in \cite{LC1},\cite{LC2}, the canonical sheaf has no non-trivial global section.\\
 
Recently, In \cite{KS} authors have studied questions related to Brill-Noether loci of moduli space of rank $2$ stable bundles over a very general quintic surface in $\mathbb P^3$ (It's another concrete example of the construction given in \cite{LC1}), in which case  the canonical sheaf not only has nontrivial global sections, moreover it's also ample and  globally generated. In \cite{DS},\cite{S}, authors have studied questions related to geometry and topology of moduli of rank $2$ stable bundles over a very general sextic surface in $\mathbb P^3$. This motivates us to extend the investigation on Brill-Noether loci in case of very general sextic surfaces in $\mathbb P^3$.\\

In \cite{HartshorneUlrich}, the authors studied Ulrich bundles on nonsingular cubic surfaces in $\mathbb P^3$ and gave  necessary and sufficient conditions on the first chern class for the existence of a stable rank $r$ Ulrich bundle with the given first Chern class. In \cite{Coskunconstruction}, the authors have shown existence of simple rank $2$ Ulrich bundles on smooth quartic surfaces in $\mathbb P^3$ and in the process a set of sufficient conditions for the existence of rank $2$ weakly Ulrich bundles are studied, also it's shown that these conditions really hold. In \cite{Coskun1} existence of every even rank stable  Ulrich bundle on a smooth quartic surface in $\mathbb P^3$ with Picard number $1$ is established. Recently in \cite{Coskunsurvey},  the author gave a  classification of the stable Ulrich bundles on smooth quartic surfaces in $\mathbb P^3$ in Picard number $2$ in terms of first Chern class. In that very survey article it's mentioned that  there are no specific results for Ulrich bundles on surfaces in $\mathbb P^3$ of degree $d \geq 5$. In \cite{Casnati} it's shown that for smooth surfaces in $\mathbb P^3$ with degree $3 \leq d \leq 9$, the surface supports families of dimension $p$ of pairwise non-isomorphic, indecomposable, Ulrich bundles for arbitrarily large $p$. Also in \cite{Chiantini5} and \cite{ACM6} rank $2$ ACM bundles on a general quintic and sextic surface are classified. This motivates us to study weakly Ulrich bundles and Ulrich bundles on a very general sextic hypersurface in $\mathbb P^3$ and its implication to the non-emptiness question of the Brill-Noether theory on  such surfaces. \\

In this pursuit one can pose the following natural problems on a very general sextic surface:\\

$(i)$ Give bounds on $c_2$ such that the corresponding Brill-Noether loci $\mathcal W^{1}_{2}$ are non-empty.\

$(ii)$ Give bounds on $c_2$ such that the Brill-Noether loci $\mathcal W^{1}_{2}$ contain a smooth point and the irreducible component containing it has the expected dimension.\

$(iii)$ Give a lower bound on $c_2$ such that the Brill-Noether loci $\mathcal W^{0}_{2}$ are non-empty and every irreducible component of each such locus has the expected dimension.\

$(iv)$ Give if possible a set of sufficient conditions  for the existence of a rank $2$  weakly Ulrich bundle which is not Ulrich and investigate whether this conditions really hold. Study possible implication of existence of weakly Ulrich bundles in terms of non-emptiness of Brill-Noether loci.\

$(v)$ Study the  implication of existence of a rank $2$ Ulrich bundle in terms of non-emptiness of Brill-Noether locus.\

$(vi)$ Study the possibility of existence of higher rank simple Ulrich bundles.\\

In addressing the first three questions  we closely follow the techniques used in the case of quintic surface (cf.\cite{KS}). To be precise we use the techniques of  sheaves (in particular rank $2$ bundles) on algebraic surface and the theory of algebraic curves to give an answer to the first question, the geometric informations of the moduli  from \cite{S} and the injectivity of Petri map on surface to address the second question. Finally using techniques related to Cayley-Bacharach property on surface and techniques from \cite{S}, \cite{Simpson11}, we answer the third question. We here mention that topological questions like connectedness, irreducibility of corresponding Brill-Noether loci can also be raised in this context, although that is not attempted here. For study related to Ulrich bundles we mainly adapt the techniques used in \cite{Coskunconstruction}, \cite{Coskun1}.\\

To be precise, our main results are the following:\\

\begin{theorem}\label{PP1}

If $103 \leq c_2 \leq 107$, then $\mathcal W^{1}_{2} - \mathcal W^{2}_{2} \subset \mathcal M(2;5H,c_2)$ is non-empty and for $c_2 =55$, $\mathcal W^{11}_2- \mathcal W^{12}_2$ is non-empty.\\

\end{theorem}

\begin{theorem} \label{PP2}

If $103 \leq c_2 \leq 107$, then  $\mathcal W^{1}_{2} \subset \mathcal M(2;5H,c_2)$ contains a smooth point and the irreducible component containing it has the actual dimension same as the expected dimension.\\

\end{theorem}

\begin{theorem} \label{PP3}

If $c_2 \geq 117$, then the Brill-Noether locus $\mathcal W^{0}_{2} \subset \mathcal M(2;5H,c_2)$ is non-empty and every irreducible component of it has the actual dimension same as the expected dimension.\\
\end{theorem}

\subsection{ Overview of the proofs and the structure of the paper:}

In section \ref{S1}, we will recall some basic results which we will use in the subsequent sections. More precisly, we will collect facts regarding coherent systems from \cite{HeMin}, we will recall necessary useful facts regarding algebraic curves from \cite{ACGH} and other sources, necessary results regarding sheaves on surface and Cayley-Bacharach property from \cite{Huy} , results regarding the geometry of moduli of rank $2$ stable bundles on a very general sextic surface in $\mathbb P^3$ from \cite{S}. Finally, we will recall  facts regarding Ulrich bundles on surfaces from \cite{Coskunconstruction}, \cite{HartshorneUlrich}, \cite{Coskundelpezzo}, \cite{Casnati} and other sources.\\
In section \ref{S2},  for convenience we briefly outline the construction of the Brill-Noether locus $\mathcal W^{r}_{k,H}$ following \cite{LC1}, \cite{KS}. We also recall the notion of Petri map for smooth, projective, irreducible algebraic variety of arbitrary dimension having canonical bundle ample following \cite{KS} and the equivalence of existence of smooth points in $\mathcal W^{k}_{r,H}$ and the injectivity of Petri map as an analogous case of curves.\\
In section \ref{S3}, using elementary techniques we first prove that if there exists a base-point free line bundle on $C \in |5H|$ with exactly $2$ sections with degree greater than or equal to $61$,  then $\mathcal W^{1}_{2} \subset \mathcal M(2;5H,c_2)$ is non-empty provided $c_2$ is same as the degree of the line bundle. Then we show the existence of a base-point free line bundle (on a curve $C \in |5H|$) of degree $103\leq d \leq 107$ with exactly $2$ sections using classical Brill-Noether theory. The candidate for non-emptiness is precisely the dual of the kernel of the evaluation map associated to the line bundle. Using these two propositions we finally prove the first part of Theorem \ref{PP1}. We also mention in the remark a generalized version of all three results proven in this section.\\
In section \ref{S4}, using certain moduli space isomorphism and the necessary condition for the existence of a non-zero Higgs field (or twisted endomorphism) in the case of sextic surface from \cite{S}, we show that for $103 \leq c_2 \leq 107$, if $\mathcal E \in \mathcal W^{1}_{2}- \mathcal W^{2}_{2}$ is same as the candidate for the non-emptiness described in previous section, then $\mathcal E$ is a smooth point of $\mathcal M(2;5H,c_2)$. Then using the injectivity of Petri map we prove Theorem \ref{PP2}. \\
In section \ref{S5}, using the Cayley-Bacharach property and other elementary techniques we  first show the non-emptiness of $\mathcal W^{0}_{2}$ for $c_2 \geq 117$, we  then prove a useful dimension estimate lemma following \cite{S}, which coupled with a dimension estimate result from \cite{Simpson11} leads to the completion of a proof of Theorem \ref{PP3}.\\
Finally in section \ref{S6}, in the first subsection following \cite{Coskunconstruction} we give a set of sufficient conditions for the existence of a rank $2$ weakly Ulrich bundle which is not Ulrich and investigate whether this conditions really hold. We obtain a partial answer to this question and mention possible implication of obtaining some more non-empty Brill-Noether loci in addition to the first part of Theorem \ref{PP1}. In the second subsection, using the existence of a rank $2$ Ulrich bundle  on a very general sextic surface we indeed obtain some more non-empty Brill-Noether loci and complete the proof of Theorem \ref{PP1}. Finally, adapting techniques from \cite{Coskun1} we investigate the possibility of existence of higher rank simple Ulrich bundles on a very general sextic surface.

\section*{Notation and convention}
We always work over the field $\mathbb C$ of complex numbers. If $X$ is a smooth, projective, irreducible variety, we denote by $\omega_X$, the canonical sheaf (canonical bundle) on $X$. For a coherent sheaf $E$ on $X$ we denote by $H^i(X,E)$ the i'th cohomology group of $E$ (which is also a $\mathbb C$ vector space) and by $h^i(X,E) := \text{dim}_{\mathbb C}H^i(X,E)$. For   the line bundle $\mathcal{O}_{\mathbb{P}^3}(n)$ we write simply   $\mathcal{O}(n)$ and for a subscheme $Y \subset \mathbb{P}^3$, we denote the pull back of $\mathcal{O}(n)$ to $Y$  by $\mathcal{O}_Y(n)$. If $\mathcal E$ is a vector bundle on $X$ (we make no distinction between vector bundle and locally free sheaf), then by $\mathcal E^* := \mathcal Hom(\mathcal E, \mathcal O_X)$ we mean the dual of the vector bundle $\mathcal E$. For a zero dimensional subscheme $Z$ of $\mathbb P^3$, we denote by $\mathcal J_Z$ the ideal sheaf of $Z$. For a coherent sheaf $E$ on $X$, we denote by $c_i(E)$  the i'th Chern class of $E$. We will make no distinction between a divisor and its associated line bundle and vice-versa. For any real number $x$, $\lfloor x \rfloor$ will denote the greatest integer smaller than or equal to $x$.\\

\section{Technical Preliminaries}\label{S1}
In this section we will recall few preparatory results which we need in next sections.\\

Let $r$ be a positive integer. Consider the linear system $\mathcal L_r = |\mathcal O(r)|$ in $\mathbb P^3$. We know that it's a projective space of dimension ${r+3 \choose 3} -1$. Let's denote by $U_r$ the dense open subset in $\mathcal L_r$, whose points correspond to smooth hypersurfaces of degree $r$ in $\mathbb P^3$.\

By Noether-Lefschetz theorem (cf.\cite{GH}), the Picard group, $Pic(X)$ of a very general surface  $X \in U_r$ with $r\geq 4$ is generated by $\mathcal O_X(1)$ (i.e $Pic(X) \cong \mathbb Z$). The phrase \textit{very general} means that the property holds outside the union (say $\mathcal N_r$) of countably many proper zariski closed subset of $U_r$. We call $\mathcal N_r$ the \textit{Noether-Lefschetz locus} in degree $r$.\\

If $X$ is a very general sextic hypersurface in $\mathbb P^3$, then by adjunction  (see \cite{RH}, example $8.20.3$), we have $\omega_X \cong \mathcal O_X(2)$. looking at the piece of the long exact cohomology sequence :\
\begin{equation}
 H^1(\mathbb P^3, \mathcal O(m)) \to H^1(X, \mathcal O_X(m)) \to H^2((\mathbb P^3, \mathcal O(m-6))
\end{equation}
and using facts from the cohomology of the projective space (cf. \cite{RH}, Chapter-$3$, Theorem $5.1(b)$), we see that $H^1(X, \mathcal O_X(m))=0, \forall m \in \mathbb Z$. More precisly, using Serre duality and above observation we have the following :\\

\begin{itemize}

\item $H^0(X, \mathcal O_X(m)) \cong H^0(\mathbb P^3, \mathcal O(m)),  \forall m \leq 5 $\\ 

\item  $H^1(X, \mathcal O_X(m)) \cong H^1(\mathbb P^3, \mathcal O(m)) = 0,  \forall m \in \mathbb Z$\\ 

\item  $H^2(X, \mathcal O_X(m)) \cong (H^0(\mathbb P^3, \mathcal O(2-m)))^*,  \forall m \geq -3$\\

\end{itemize}

\begin{remark}
This shows that the canonical sheaf has non-trivial global section. It's also ample and globally generated.
\end{remark}

For the convenience of computation in later parts of the article we make the following remark on some numerical invariants (first and second chern classes, euler characteristic) of a vector bundle.\\

\begin{remark}

Let $X \subset \mathbb P^3$ be a surface of degree $d$  and $\mathcal E$ be a rank $r$ vector bundle on $X$, then we have the following :

\begin{equation}{\label{e1}}
c_1(\mathcal E(m)) = c_1(\mathcal E) + rm,
\end{equation}
\begin{equation}{\label{e2}}
c_2(\mathcal E(m)) = c_2(\mathcal E) + dm(r-1)c_1(\mathcal E) + dm^2.{r \choose 2}
\end{equation}
The Euler characteristic of a rank $2$ vector bundle can be computed by Riemann-Roch formula for vector bundles on an algebraic surface:\
 
$\chi(\mathcal E) = 2.\chi(\mathcal O_X) + \frac{1}{2}(c_{1}^{2}(\mathcal E) - c_1(\mathcal E).c_1(\omega_X)) -c_2(\mathcal E) $. \\

\end{remark}

Let $X$ be a smooth, projective, irreducible variety of dimension $n$ with an ample divisor $H$ on it. Let $E$ be a torsion-free sheaf (cf.\cite{Huy}) on $X$. Then from Hirzebruch-Riemann-Roch formula, we have $\text{deg}_H(E) = c_1(E).H^{n-1}$. The slope of $E$ with respect to $H$ is defined as $\mu_H(E) = \frac{\text{deg}_H(E)}{r(E)}$, where $r(E) : =$ rank of $E$ (see \cite{Huy}).\\

\begin{definition}\label{D1}(Mumford,Takemoto / slope (semi)stability)\\
$E$ is called $\mu_H$ semistable (respectively stable) if for any nontrivial, proper  subsheaf $F \subset E$ with $r(F) < r(E)$ we have, $\mu_H(F) \leq \mu_H(E)$(respectively, $\mu_H(F) < \mu_H(E))$.\
\end{definition}


\subsection{Coherent Systems}
    
Here we recall some basic definition and results on coherent system. Interested readers may look at \cite{HeMin}, \cite{LP} for a detailed account.\\

\begin{definition}\label{D2}(D\'{e}finition $1.1$, \cite{HeMin})\\
An algebraic system on an algebraic variety $X$ is a triplet $\Lambda = (\Gamma, \sigma, F)$, where $F$ is an $\mathcal O_X$ module, $\Gamma$ is a $\mathbb C$ vector space and $\sigma : \Gamma \to H^0(F)$ is a $\mathbb C$ linear map.\\
A morphism between $2$ algebraic systems $\Lambda = (\Gamma, \sigma, F)$ and $\Lambda^{'} =(\Gamma^{'}, \sigma^{'}, F^{'})$ is a pair of morphisms denoted by $\Theta =(\alpha,\beta) : \Lambda \to \Lambda^{'}$, where $\beta :F \to F^{'}$ is a morphism of $\mathcal O_X$ modules and $\alpha : \Gamma \to \Gamma^{'}$ is a $\mathbb C$ linear map such that $\sigma^{'}\alpha=H^0(\beta)\sigma$.\\
An algebraic system $\Lambda = (\Gamma, \sigma, F)$ is said to be a Coherent system if $F$ is a Coherent $\mathcal O_X$ module and the dimension of $\Gamma$ is finite.\\
\end{definition}

Given a Coherent system $\Lambda$, we denote by $Ext^i(\Lambda,-)$ the derived functors of the functor $Hom(\Lambda,-)$ and similarly by $\mathcal Ext^i(\Lambda,-)$ the derived functors of the functor $\mathcal Hom(\Lambda,-)$.\\

\begin{proposition}\label{P1}
Let $f :X \to S$ be a flat morphism and  $\Lambda =(\Gamma,\sigma,F)$ and $\Lambda^{'} =(\Gamma^{'},\sigma^{'},F^{'})$ be two algebraic systems on $X/S$. Then there is a long exact sequence:\
\begin{align*}
& 0 \to \mathcal Hom_{f}(\Lambda,\Lambda^{'}) \to \mathcal Hom(\Gamma,\Gamma^{'}) \oplus \mathcal Hom_{f}(F,F^{'}) \to \mathcal Hom_{f}(f^{*}(\Gamma), F^{'}) \to \mathcal {E}xt^{1}_{f}(\Lambda,\Lambda^{'})  \\
& \to \mathcal Ext^1(\Gamma, \Gamma^{'}) \oplus \mathcal Ext^1_{f}(F,F^{'}) \to \mathcal Ext^1_{f}(f^{*}(\Gamma), F^{'}) \to \cdots
\end{align*}
\end{proposition}

\begin{proof}
See Proposition $1.5$, \cite{HeMin}.
\end{proof}
   
We  state  a useful corollary of the above Proposition which will be later used in the construction of \textit{Petri map}.\\

\begin{corollary}\label{C1}(\cite{HeMin}, Corollaire $1.6)$\\
If $\Lambda =(\Gamma,\sigma,F)$ and $\Lambda^{'} =(\Gamma^{'},\sigma^{'},F^{'})$ are $2$ algebraic systems with $\sigma^{'}$ injective, then there exists an exact sequence :\
\begin{align*}
& 0 \to Hom(\Lambda,\Lambda^{'}) \to Hom(F,F^{'}) \to Hom(\Gamma, H^0(F^{'})/\Gamma^{'})\\
& \to Ext^{1}(\Lambda,\Lambda^{'}) \to Ext^1(F,F^{'}) \to Hom(\Gamma,H^1(F^{'}))\\
& \to Ext^2(\Lambda,\Lambda^{'}) \to Ext^2(F,F^{'}) \to Hom(\Gamma, H^2(F^{'})) \to \cdots
\end{align*}
\end{corollary}

Next we recall $\alpha$-semistability (stability) of a Coherent system.\

Consider a polynomial $\alpha >0$ with rational coefficients. For a coherent system $\Lambda =(\Gamma,\sigma,F)$ of dimension $d$ and multiplicity $r$, the reduced Hilbert polynomial $\rho^{\alpha}_{\Lambda}$ relative to $\alpha$ of the coherent system $\Lambda$ is defined by, $\rho^{\alpha}_{\Lambda} =\frac{\text{dim}\Gamma}{r}.\alpha+\frac{P_F}{r}$, where $P_F$ is the Hilbert polynomial of $F$. Once $\alpha$ is fixed we just write $\rho_{\Lambda}$. Note that, the dimension of $\Lambda$ is $\text{dim}(\Gamma)$ and the multiplicity of $\Lambda$ is the multiplicity of the coherent sheaf $F$.\

\begin{definition}{\label{D3}}($\alpha$-semistability(stability))\\
A Coherent system $\Lambda =(\Gamma,\sigma,F)$ of dimension $d$ is $\alpha$-semistable (stable) if,\\
$(i)$ the Coherent system $\Lambda =(\Gamma,\sigma,F)$ is pure (i.e $F$ is pure),\\
$(ii)$ for proper non-zero sub-Coherent systems $\Lambda^{'} =(\Gamma^{'},\sigma^{'},F^{'}) \subset \Lambda$, one has\
 $\rho_{\Lambda^{'}} \leq \rho_{\Lambda}$ (respectively, $<$).\\
\end{definition}

Finally, we are in a position to mention the \textit{smoothness criterion} which will be used in section $3$ (and hence indirectly in section $5$). For the construction of the variety $\text{Syst}_{X,\alpha}(k,P)$ (see \cite{HeMin}, section $2.2$ and $2.3$).\\

\begin{theorem}\label{T2}(Smoothness Criterion)\\
Let $\Lambda_0$ be an $\alpha$-stable coherent system and has the corresponding point in $\text{Syst}_{X,\alpha}(k,P)$. Then the tangent space at $a$ is isomorphic to $Ext^{1}_{\mathbb C}(\Lambda_0,\Lambda_0)$ and if $Ext^{2}_{\mathbb C}(\Lambda_0,\Lambda_0) =0$, then the variety $\text{Syst}_{X,\alpha}(k,P)$ is smooth at $a$.
\end{theorem} 

\begin{proof}
See  \cite{HeMin}, Th\'{e}or\`eme $3.12$.
\end{proof}

\subsection{Sheaves on surfaces}
   
Here we list some basic useful facts about sheaves on surface which will be used in subsequent sections. Throughout this subsection $X$ denotes a smooth, complex, projective surface unless mentioned otherwise.\\

\begin{definition}\label{D3}
Let $C$ be an effective divisor on  $X$. If $\mathcal F$ and $\mathcal G$ are vector bundles on $X$ and $C$, respectively, then a vector bundle $\mathcal E$ on $X$ is obtained by an elementary transformation of $\mathcal F$ along $\mathcal G$ if there exists an exact sequence of the following form:\
\begin{align*}
0 \to \mathcal E \to \mathcal F \to i_{*}\mathcal G \to 0
\end{align*}
where, $i$ denotes the embedding $C \subset X$.\\
\end{definition}

The following Proposition is useful in calculating the first and second Chern classes of the dual of the kernel of the evaluation maps, appearing in section \ref{S3} and \ref{S6}.\\

\begin{proposition}\label{P2}
If $\mathcal F$ and $\mathcal G$ are locally free on $X$ and $C$, respectively, then the kernel $\mathcal E$ of any surjection $\phi :\mathcal F \to i_{*}\mathcal G$ is locally free. Moreover, if $\rho$ denotes the rank of $\mathcal G$, one has $\text{det}(\mathcal E) \cong \text{det}(\mathcal F) \otimes \mathcal O_X(-\rho.C)$ and $c_2(\mathcal E) = c_2(\mathcal F) - \rho C.c_1(\mathcal F)+\frac{1}{2}\rho C(\rho C + \omega_X) +\chi(\mathcal G)$.
\end{proposition} 

\begin{proof}
See \cite{Huy},Proposition $5.2.2$.
\end{proof}

Also for a smooth (or at least reduced) curve $C$, the characteristic $\chi(\mathcal G)$ can be written as $\chi(\mathcal G) = \text{deg}(\mathcal G) - \frac{\rho}{2}C.(\omega_X + C)$. Hence $c_2(\mathcal E) = c_2(\mathcal F) + (\text{deg}(\mathcal G) -\rho C .c_1(\mathcal F)) + \frac{\rho(\rho-1)}{2}C^{2}$.\\

For a detailed account of zero-dimensional subschemes see \cite{LG0}. We define the \textit{length of a zero dimensional scheme} $Z$ as $l(Z) = h^0(\mathcal O_Z)$.\\

\begin{definition}(Cayley-Bacharach property)
A zero dimensional subscheme $Z \subset  X$ is said to satisfy Cayley-Bacharach (CB in short) for a line bundle $L$, if for any subscheme $Z^{'} \subset Z$  with $l(Z^{'}) = l(Z) -1$ and $s\in H^0(L)$ with $s|_{Z^{'}}=0$, then $s|_{Z} = 0$.\\
\end{definition}

It's not difficult to see that a zero-dimensional subscheme  $Z \subset X$ which satisfies Cayley-Bacharach property for $\mathcal{O}_X(m)$  also satisfies Cayley-Bacharach property for $\mathcal{O}(m)$   and vice-versa. Thus, if  a zero-dimensional subscheme  $Z \subset X$ satisfies Cayley-Bacharach property for $\mathcal{O}_X(m)$,  then with out loss of generality we can assume  that  $Z$  satisfies Cayley-Bacharach property for $\mathcal{O}(m)$ in $\mathbb{P}^3$ and  we write $Z$ satisfies $\text{CB}(m)$.\\

The following theorem gives us an important implication of when a $0$-dimensional locally complete intersection subscheme (l.c.i in short)[see \cite{RH}, Page-$185$] satisfies Cayley-Bacharach property for some line bundle in terms of local freeness of some torsion free sheaves which we will use in section \ref{S5}.\\

\begin{theorem}\label{CB}
Let $Z \subset X$ be a local complete intersection of codimension $2$ and let $L$ and $M$ be line bundles on $X$. Then there exists an extension:\
\begin{align*}
0 \to L \to E \to M \otimes \mathcal J_Z \to 0
\end{align*}
such that $E$ is locally free if and only if $Z$ satisfies Cayley-Bacharach property for the line bundle $L^*\otimes M \otimes \omega_X$.
\end{theorem}

\begin{proof}
See \cite{Huy}, Theorem $5.1.1$.
\end{proof}

let $\mathcal E$ be a rank $2$ stable bundle on $X$, we define $\text{End}^0(\mathcal E)$ as the kernel of the trace map from $\text{End}(\mathcal E)$ to $\mathcal O_X$. In notation, $\text{End}^{0}(\mathcal E) = \text{ker}(\text{tr} : \text{End}(\mathcal E) \to \mathcal O_X)$. In what follows we list a necessary condition for a rank $2$ stable bundle to be potentially obstructed followed by a dimension estimate result which will be used in proving smoothness of a point inside moduli space in section \ref{S4}.\\

\begin{proposition}\label{p8}
Let $X \subset \mathbb P^3$ be a very general sextic surface. Let $\mathcal E$ be a  rank $2$ stable bundle on $X$ (w.r.to $\mathcal O_X(1)$) with $\text{det}(\mathcal E) \cong \mathcal O_X(1)$ and $\mathcal E$ is potentially obstructed, i.e,  $H^0(X, \text{End}^0(\mathcal E) \otimes \omega_X) \neq 0$. Then either : \

$(i)$ $H^0(X, \mathcal E) \neq 0$ \

 OR\
 
$(ii)$ there is a section $\beta \in H^0(X, \omega^2_X)$ which is not square defining a double cover $r : Z \to X$ with $Z \subset \omega_X$ and $r$ is ramified along zero$(\beta)$, together with a line bundle $L$ over a desingularization $\epsilon :\tilde Z \to Z$ such that $\mathcal E = r_*\epsilon_*(L)^{**}$.

\end{proposition}

\begin{proof}
See \cite{S}, Proposition $2.2$
\end{proof}

\begin{proposition}\label{p9}
The dimension of the locus of the potentially obstructed bundles of type $(ii)$ is $ \leq 39$.
\end{proposition}

\begin{proof}
See \cite{S}, Lemma $2.4$ and the discussion before that.
\end{proof}

\subsection{Line bundles on curves}

Here we recall some basic results from the classical Brill-Noether theory of line bundles on curves. For an elaborate account see \cite{ACGH}.\

Let $C$ be a smooth, projective curve of genus $g$. For $r \geq 0,d \geq 1$ fixed integers, we denote, $\text{Supp}(W^{r}_{d}(C))=\{L \in Pic^d(C) | h^0(C,L) \geq r+1\}$. The following theorem provides a lowerbound for the dimension of $W^{r}_{d}(C)$.

\begin{theorem}\label{T2.15}
Let $\rho = \rho(g,r,d)$ be the Brill Noether number $\rho : = g -(r+1)(g+r-d)$. If $\rho \geq 0$, then $W^r_d(C)$ is non-empty. Furthermore, if $g-d+r \geq 0$, then each irreducible component of $W^r_d(C)$ has dimension atleast $\rho$.
\end{theorem}

\begin{proof}
See \cite{ACGH}, Chapter-$V$, Theorem $1.1$.\\
\end{proof}

We always have $W^{r+1}_d(C) \subseteq W^r_d(C)$ for all $r \geq 0$ and for all $d \geq 1$. The next lemma gives  a sufficient condition for the inclusion to be strict.\\

\begin{lemma}\label{L2.16}
Suppose $g+r-d \geq 0$. Then no component of $W^r_d(C)$ is entirely contained in $W^{r+1}_d(C)$.
\end{lemma}

\begin{proof}
See \cite{ACGH}, Chapter-$IV$, Lemma $3.5$.
\end{proof}





\begin{definition}
A curve $C$ is called hyperelliptic if its genus $g \geq 2$ and there exists a finite morphism $f : C \to \mathbb P^1$ of degree $2$.\\
\end{definition}

The following theorem gives an upperbound on $W^r_d(C)$, for a non-hyperelliptic curve $C$ with certain restrictions on $r,g,d$. All the following results will be useful in proving existence of base-point free line bundles in section \ref{S3}.\\

\begin{theorem}\label{p17}(Martens theorem)
Let $C$ be a smooth curve of genus $g \geq 3$. Let $d$ be an integer such that $2 \leq d \leq g-1$ and $r$ be an integer such that $0 <2r \leq d$. Then if $C$ is not hyperelliptic,  every component of $W^r_d(C)$ has dimension at most $d-2r-1$. In symbols $\text{dim}(W^r_d(C)) \leq d-2r-1$.\
If $C$ is hyperelliptic, then $\text{dim}(W^r_d)(C) = d-2r$.
\end{theorem}

\begin{proof}
See \cite{ACGH}, Chapter-$IV$, Theorem $5.1$.\\
\end{proof}

We shall say a smooth curve $C$ is \textit{bi-elliptic} if it can be represented as a ramified double covering of an elliptic curve. The following theorem is a refinement of Marten's theorem.\\

\begin{theorem}\label{p18}(Mumford's theorem)
Let $C$ be a smooth non-hyperelliptic curve of genus $g \geq 4$. Suppose that there exists integers $r$ and $d$ such that $2 \leq d \leq g-2$, $d \geq 2r >0$ and a component $X$ of $W^r_d(C)$ such that $\text{dim}(X) = d-2r-1$, then $C$ is either a trigonal, bi-elliptic or a smooth plane quintic.
\end{theorem}

\begin{proof}
See \cite{ACGH}, Chapter-$IV$, Theorem $5.2$.\\
\end{proof}

\begin{definition}
Let $r,d$ be two fixed non-negative integers. A line bundle $L$ on $C$ is called a complete $g^r_d$ on $C$ if $L \in W^r_d(C) -W^{r+1}_d(C)$, i.e it has exactly $r+1$ sections.\\
\end{definition}

The following theorem is an extension of Mumford's theorem. For a reference see \cite{ACGH}, Page-$200$.\\

\begin{theorem}\label{p20}(Keem's theorem)
Let $C$ be a smooth algebraic curve of genus $g \geq 11$, and suppose that for some integers $d$ and $r$  satisfying $d \leq g+r -4$, $r \geq 1$, we have $\text{dim}(W^r_d(C)) \geq d-2r-2$, then $C$  possesses a $g^1_4$.\\
\end{theorem}

\begin{theorem}\label{p21}(Harris,\cite{ACGH})
Let $C$ be a non-hyperelliptic curve of genus $g\geq 4$. Then there exists a base-point free complete pencil $g^1_{g-1}$ on $C$.
\end{theorem}

\begin{proof}
See \cite{KC}, Theorem $1.4$.\\
\end{proof}

\begin{proposition}\label{p22}(\cite{ACGH}, Page-$139$, $C-4$)
Let $C \subset \mathbb P^3$ be a smooth complete intersection of two surfaces of degree $m$ and $n$, then $C$ does not have a $g^{1}_{m+n-3}$.\\
\end{proposition}

\begin{proposition}\label{p26}(\cite{ACGH}, Page-$221$, $B-4$)
If $C$ is a smooth bi-elliptic curve of genus $g \geq 6$ and $D$ is a general effective divisor of degree $d$ on $C$, then $\phi_{K(-D)}$ is not an embedding unless $d\leq 1$.\\
\end{proposition}

\subsection{Ulrich bundles on surfaces}
Here we collect some basic results on Ulrich bundles on smooth projective surfaces which will be used in  section \ref{S6}. For a detailed survey on Ulrich bundles see \cite{BIU}, \cite{Coskunsurvey}. We start by recalling the definitions of the basic objects concerned.\\

\begin{definition}\label{d24}(Weakly Ulrich bundle)
Let $(X, \mathcal O_X(1))$ be a smooth, polarized projective variety of $\text{dim}(X) = n$ embedded in $\mathbb P^N$, then a vector bundle $\mathcal E$  on $X$ is said to be an Weakly Ulrich bundle  if the following conditions  are satisfied :\

$(i)$ $H^i(\mathcal E(-m)) =0$ for $1 \leq i \leq n$ and $m \leq i-1 \leq n-1$ and \

$(ii)$ $H^i(\mathcal E(-m)) =0$ for $0 \leq i \leq n-1$ and $m \geq i+2$.\\

\end{definition}

\begin{definition}\label{d25}(Ulrich bundle)
With the same notation as before, we say that a vector bundle $\mathcal E$ on $X$ is an Ulrich bundle if the following conditions hold :\

$(i)$ $H^n(\mathcal E(-m)) =0$ for all $m \leq n$,\

$(ii)$ $H^i(\mathcal E(-m)) =0$ for $1 \leq i \leq n-1$ and $ \forall m \in \mathbb Z$ (This condition is same as $\mathcal E$ being arithmatically Cohen-Macaulay) and \

$(iii)$ $H^0(\mathcal E(-m)) =0$, $\forall m \geq 1$.\\
\end{definition}

\begin{definition}(Arithmetically Gorenstein surfaces)
A smooth projective variety $X \subseteq \mathbb P^n$ is called AG (Arithmetically Gorenstein) if it is ACM \footnote{A smooth projective variety $X \subseteq \mathbb P^n$ is called ACM (Arithmetically Cohen-Macaulay if its homogeneous coordinate ring $S_X$ is Cohen-Macaulay, or equivalently, if $\text{dim}(S_X) = \text{depth}(S_X)$. }and its canonical bundle is isomorphic to $\mathcal O_X(m)$ for some $m \in \mathbb Z$.\\
\end{definition}

The following proposition gives  necessary and sufficient conditions for a vector  bundle $\mathcal E$ to be an Ulrich bundle on a AG surface of degree atleast $2$ in $\mathbb P^n$.\\

\begin{proposition}\label{p2.30}
Suppose $X \subseteq \mathbb P^n$ is an AG surface of degree $d \geq 2$ with hyperplane class $H$ and $\omega_X = mH$. Then a vector bundle $\mathcal E$ of rank $r$ on $X$ is Ulrich if and only if the following conditions hold :\

$(i)$ $\mathcal E$ is ACM\footnote{A vector bundle $\mathcal E$ on a projective variety $X$ is ACM if all its intermediate cohomologies vanishes, i.e,  $H^i(X, \mathcal E(m)) = 0$, for $0 <i <\text{dim}(X)$ and all $m \in \mathbb Z$}.\

$(ii)$ $c_1(\mathcal E). H = (\frac{m+3}{2}). dr$.\

$(iii)$ $c_2(\mathcal E) = \frac{c_1(\mathcal E)^2}{2} - dr.\frac{m^2 +3m+4}{4} + r.(1 + h^0(\omega_X))$.

\end{proposition}

\begin{proof}
See \cite{Coskundelpezzo}, Proposition $2.10$.
\end{proof}

\begin{remark}\label{r2.31}
It's not very difficult to see that a very general surface of degree $d \geq 5$ in $\mathbb P^3$ is an AG surface. In particular, for a very general sextic surface $X\subset \mathbb P^3$ a bundle $\mathcal E$ of rank $2$ on $X$ is Ulrich iff It's ACM, $c_1(\mathcal E) =5H, c_2(\mathcal E) = 55$.\\
\end{remark}

The following two results are on the existence of a rank $2$ Ulrich bundle on a smooth surface in $\mathbb P^3$.\\

\begin{proposition}\label{P2.32}
A general surface of degree $d$ in $\mathbb P^3$ can be defined by a linear pfaffian\footnote{We say that $X \subset \mathbb P^3$ of degree $d$ can be defined by a linear pfaffian if there exists a $(2d)\times (2d)$ skew symmetric matrix $M$ with linear entries such that $X =\{\text{pf}(M)=0\} \subset \mathbb P^3$} if and only if $d \leq 15$.
\end{proposition}

\begin{proof}
See \cite{Beauvillehypersurface}, Proposition $7.6(b)$.
\end{proof}

\begin{lemma}\label{L2.33}
Let $(X, \mathcal O_X(1))$ be a smooth polarized surface of degree $d$ in $\mathbb P^3$. Then $X$ is a linear pfaffian if and only if it is endowed with a rank $2$ Ulrich bundle $\mathcal E$ such that $c_1(\mathcal E) = \mathcal O_X(d-1)$
\end{lemma}
  
\begin{proof}
See \cite{Casnati}, Lemma $3.1$.\\
\end{proof}

The following two results deal with the simplicity and stability properties of an Ulrich bundle.\\

\begin{proposition}\label{p2.34}
Let $(X,\mathcal O_X(1))$ be a smooth, polarized projective variety. Then the following are true :\

$(i)$ An Ulrich bundle $\mathcal E$ is both Gieseker-semistable and slope semistable. \

$(ii)$ If $\mathcal E$ is strictly semistable of rank $r \geq 2$, then it's destabilized by an Ulrich subbundle of strictly smaller rank.\

$(iii)$  For an Ulrich bundle Gieseker stability and slope stability are equivalent.
\end{proposition}

\begin{proof}
Statements $(i),(iii)$ follows from \cite{HartshorneUlrich}, Theorem $2.9(a),(c)$. For statement $(ii)$ see \cite{Coskunternary}, Lemma $2.15$.
\end{proof}

\begin{lemma}\label{simple}
On a nonsingular projective variety $X$, let
\begin{align*}
0 \to \mathcal F \to \mathcal E \to \mathcal G \to 0
\end{align*}
be a non-split extension of non-isomorphic $\mu_H$ stable vector bundles of same slope. Then $\mathcal E$ is a simple vector bundle.
\end{lemma}

\begin{proof}
See \cite{HartshorneUlrich}, Lemma $4.2$.\\
\end{proof}

The following Proposition is useful in computing dimensions of families of extensions.\\

\begin{proposition}\label{euler}
Let $X$ be a nonsingular projective surface of degree $d$. Let $\mathcal F_1, \mathcal F_2$ be two Ulrich bundles with rank $r_1,r_2$ respectively and first Chern classes $c_1, c_2$ respectively. Then
\begin{align*}
\chi(\mathcal F_1 \otimes \mathcal F^*_2) = r_1.c_2.\omega_X -c_1.c_2 + r_1.r_2(2d-1-p_a)
\end{align*}
where $p_a$ is the arithmetic genus of $X$.
\end{proposition}

\begin{proof}
See \cite{HartshorneUlrich}, Proposition $2.12$.\\
\end{proof}

\begin{definition}(Special Ulrich bundle)
We say that a rank $2$ Ulrich bundle $\mathcal E$ on a smooth projective variety $X$ of $\text{dim}(X) =n$ is special if $\text{det}(\mathcal E) \cong \omega_X(n+1)$.\\
\end{definition}
 
The following result deals with the construction of a rank $2$ special Ulrich bundles on a smooth projective surface. $H$ here denotes the hyperplane divisor.\\

\begin{proposition}\label{p2.36}
$(a)$ Let $C$ be a smooth curve on $X$ of class $3H +\omega_X$. Let $L$ be a line bundle on $C$ with $\text{deg}(L)= \frac{1}{2}H. (5H +3\omega_X) + 2. \chi(\mathcal O_X)$. If $\sigma_0,\sigma_1 \in H^0(L)$ define a base-point free pencil and $H^1(L(H+\omega_X)) =0$, then the bundle $\mathcal E$ defined by the \textit{Mukai exact sequence}
\begin{center}
$ 0 \to \mathcal E^* \to \mathcal O^2_X \xrightarrow{(\sigma_0,\sigma_1)} L \to 0$
\end{center} 
is a special rank $2$ Ulrich bundle.\

$(b)$ Every special rank $2$ Ulrich bundle on $X$ can be obtained from a Mukai sequence as in part $(a)$.
\end{proposition} 

\begin{proof}
See \cite{Eisenbud}, Proposition $6.2$.
\end{proof}

\section{Construction of the Brill-Noether Loci and the Petri map on higher dimensional varieties}\label{S2}

\subsection{Brill-Noether Loci}

Let $X$ be a smooth, projective, irreducible  variety of dimension $n \geq 2$ over $\mathbb C$  and $H$ be an ample divisor on $X$. Let $\mathcal M_{H} : = \mathcal M_{X,H}(r;c_1,...,c_s)$ be the moduli space of rank $r \geq 2$, $\mu_H$ stable vector bundles $\mathcal E$ on $X$ with fixed Chern classes $c_i(\mathcal E) = c_i \in H^{2i}(X,\mathbb Z)$ for $i =1,...,s : =\text{min}\{r,n\}$. The main goal of this section is to briefly recall the construction of a subvariety $\mathcal W^k_{r,H}(c_1,...c_s)$(following  \cite{LC1},\cite{KS}) of $\mathcal M_H$ such that the support of the variety is $\text{Supp}(\mathcal W^k_{r,H}(c_1,...c_s)) =\{\mathcal E \in \mathcal M_H | h^0(\mathcal E) \geq k+1\}$.\\

\begin{definition}(Determinantal Variety)
Given $m$ and $n$ and $r < \text{min}(m,n)$, the determinantal variety $Y_r$ is the set of all $m\times n$ matrices (over a field $\mathbb K$)  with rank $\leq r$. This is naturally an algebraic variety as the condition that a matrix have rank $\leq r$ is given by the vanishing of all of its $(r+1)\times(r+1)$ minors. Since the equations defining minors are homogeneous, one can consider $Y_r$ as either an affine variety in $mn$ dimensional affine space or as a projective variety in $mn-1$ dimensional projective space.\\
\end{definition}

For results on determinantal variety see \cite{HarrisAG} .\\

We only outline a rough sketch of the construction of Brill-Noether loci as we do not use the construction directly in the rest of the paper. For an explanation as to why the cohomological assumptions of the following theorem are natural see \cite{LC1}, Remark $2.4. (1)$.\\

\begin{theorem}\label{T4}
With the notations introduced as above, assume for all $\mathcal E \in \mathcal M_H$, $H^i(X,\mathcal E) = 0$ for $i \geq 2$. Then for any $k \geq -1$, there exists a determinantal variety $\mathcal W^k_{r,H}(c_1,...c_s) \subset \mathcal M_H$ such that 
\begin{center}
$\text{Support}(\mathcal W^k_{r,H}(c_1,...c_s)) =\{\mathcal E \in \mathcal M_H | h^0(\mathcal E) \geq k+1\}$
\end{center}
Moreover, each non-empty irreducible component of $\mathcal W^k_{r,H}(c_1,...c_s)$ has dimension atleast
\begin{center}
$\text{dim}(\mathcal M_H)-(k+1)(k+1-\chi(r;c_1,...c_s))$
\end{center}
where, $\chi(r;c_1,...c_s)) = \chi(\mathcal E_t)$, for any $t \in \mathcal M_H$ and 
\begin{center}
$\mathcal W^{k+1}_{r,H}(c_1,...c_s) \subset \text{Sing}(\mathcal W^k_{r,H}(c_1,...c_s))$
\end{center}
whenever $\mathcal W^k_{r,H}(c_1,...c_s) \neq \mathcal M_H$.
\end{theorem}

\begin{proof}
We first consider the case when $\mathcal M_H$ is a fine moduli space. Let $\mathcal U \to X \times \mathcal M_H$ be a universal family such that for any $t \in \mathcal M_H$, $\mathcal E_t := \mathcal U|_{X\times\{t\}} $ is a $\mu_H$ stable, rank $r$ vector bundle on $X$ with $c_i(\mathcal E_t) =c_i$. Let's  choose an effective divisor $D$ on $X$ such that for any $t \in \mathcal M_H$ we have, $h^0(\mathcal E_t(D)) = \chi(\mathcal E_t(D))$ and $H^i(\mathcal E_t(D)) =0 $ for all $i \geq 1$. Let $\mathcal D := D \times \mathcal M_H$ be the product divisor on $X \times \mathcal M_H$ and let $\pi_2 : X \times \mathcal M_H \to \mathcal M_H$ be the natural projection. From cohomological conditions on $\mathcal E_t(D)$ and the base change theorem it can be seen that, $(\pi_{2})_{*}(\mathcal U(\mathcal D))$ is a vector bundle of rank $\chi(\mathcal E_t(D))$ on $\mathcal M_H$ and $R^i(\pi_{2})_{*}(\mathcal U(\mathcal D)) =0, \forall i >0$. We have a short exact sequence on the product space $X \times \mathcal M_H$ as
\begin{center}
$0 \to \mathcal U \to \mathcal U(\mathcal D) \to \mathcal U(\mathcal D)/\mathcal U \to 0$
\end{center}
Which induces an exact sequence on $\mathcal M_H$ as follows :
\begin{center}
$0 \to (\pi_2)_{*}\mathcal U \to (\pi_2)_{*}\mathcal U(\mathcal D) \xrightarrow{\gamma} (\pi_2)_{*}\mathcal U(\mathcal D)/\mathcal U \to R^1(\pi_2)_{*}\mathcal U \to 0$
\end{center}
Here $\gamma$ is a map between two vecror bundles of rank $\chi(\mathcal E_t(D))$ and $\chi(\mathcal E_t(D)) - \chi(\mathcal E_t)$. For integer $k \geq -1$, let $\mathcal W^K_{r,H}(c_1,..c_s)$ be the $\chi(\mathcal E_t(D)) -(k+1)$ th determinantal variety (corresponding to the map $\gamma$). For any $\mathcal E_t \in \mathcal M_H$, the assumption $h^i(\mathcal E_t) =0, \forall i \geq 2$ implies $h^1(\mathcal E_t) = h^0(\mathcal E_t) - \chi(\mathcal E_t)$. Then we have
\begin{align*}
\text{Supp}(\mathcal W^K_{r,H}(c_1,...,c_s)) &= \{\mathcal E \in \mathcal M_H | h^{1}(\mathcal E) \geq (k+1) -\chi(\mathcal E)\}\\
&=\{\mathcal E \in \mathcal M_H| h^0(\mathcal E) \geq k+1\}.
\end{align*}
If $\mathcal M_H$ is not a fine moduli space then first the construction is carried out locally (the independence of choice of the local universal sheaf can be shown) and then the local construction can be glued in order to get the global object.\\
\end{proof} 

\begin{remark}
 The variety $\mathcal W^{k}_{r,H}(c_1,...c_s)$ is called the \textit{k'th Brill-Noether locus} of the moduli space $\mathcal M_H$ and the number $\rho^k_{r,X,H} := \text{dim}(\mathcal M_H)-(k+1)(k+1-\chi(r;c_1,...c_s))$ is called the \textit{generalized Brill-Noether number} and it's also called the expected dimension of  $\mathcal W^{k}_{r,H}(c_1,...c_s)$. From the preceding theorem we see that under non-emptiness assumption every irreducible component of $\mathcal W^{k}_{r,H}(c_1,...c_s)$ and hence $\mathcal W^{k}_{r,H}(c_1,...c_s)$ has dimension atleast $\rho^k_{r,X,H}$. When the context is clear we will just write $\mathcal W^k_r$ , $\rho^k_r$ instead of  $\mathcal W^{k}_{r,H}(c_1,...c_s)$ and $\rho^k_{r,X,H}$.
\end{remark}

\subsection{Petri map for higher dimensional varieties}

Here we recall the definition of Petri map for higher dimensional varieties following \cite{KS}.\\

\begin{definition}\label{petri}
Let $X$ be a smooth, projective, irreducible variety of dimension $n$ and $\mathcal E$ be a vector bundle on $X$. Let's consider two coherent systems on $X$ given by  $\Lambda = \Lambda^{'} =(H^0(X,\mathcal E), i_d, \mathcal E)$. Then from Corollary \ref{C1}, we have a map: \
\begin{center}
$Ext^1(\mathcal E, \mathcal E) \to Hom( H^0(X, \mathcal E), H^1(X, \mathcal E))$
\end{center} 
Using Serre duality, we have the corresponding dual map
\begin{equation}\label{Petri}
\mu : H^0(X,\mathcal E)\otimes H^{n-1}(X, \omega_X \otimes \mathcal E^*) \to H^{n-1}(\omega_X \otimes \mathcal E \otimes \mathcal E^{*})
\end{equation}
We call this cup product map as the \textit{Petri map}.\\
\end{definition}

\begin{remark}
Let $n=2$ and $\mathcal E \in \mathcal M_H$ be a smooth point of the moduli space such that $\mathcal E \in \mathcal W^k_r -\mathcal W^{k+1}_r$(i.e it has exactly $k+1$ sections). Let $T$ be the tangent space of $\mathcal M_H$ at the point $\mathcal E$. Then $T = \text{Im}(\mu)^{\perp}$ and we have,
\begin{center}
$\text{dim}(T) = \text{dim}(\mathcal M_H)-h^0(X,\mathcal E).h^1(X,\mathcal E) + \text{dim ker}(\mu)$.
\end{center}
If $\mu$ is injective, then $\mathcal E$ is a smooth point of $\mathcal W^k_r$ and the component containing it has the expected dimension. More precisely, one has from Corollary \ref{C1} and Theorem \ref{T2} that $\mathcal E$ is a smooth point in  $\mathcal W^k_r$ if and only if the Petri map($\mu$) is injective. One can also define the \textit{Petri map} for higher dimensional varieties along the lines of \cite{alternativepetri}, section $2$.

\end{remark}

\section{Non-emptiness of Brill-Noether loci $\mathcal W^1_2 -\mathcal W^2_2$}\label{S3}

From now on, we concentrate on the case when $X$ is a very general sextic hypersurface in $\mathbb P^3$. Let $H: = c_1(\mathcal O_X(1))$. Consider $\mathcal M(c_2) := \mathcal M_{X,H}(2;5H,c_2)$ to be the moduli space of rank $2$, $\mu_H$-stable vector bundles $\mathcal E$ on $X$ with $c_1(\mathcal E) = 5H$ and $c_2(\mathcal E) = c_2$. Note that in this situation, we have $\chi(\mathcal O_X) = 11$ and hence the expected dimension of the moduli space $\mathcal M(c_2)$ is $4c_2 - c^2_1 - 3\chi(\mathcal O_X) = 4c_2 - 150 -3.11 = 4c_2 -183$.\\

 Recently in \cite{S}, it's shown that $\mathcal M(2;H,c_2)$
 is good (i.e generically smooth of the  expected dimension) for $c_2 \geq 20$ and  irreducible  for $c_2 \geq 27$. Also from \cite{Simpson18} it's known that $\mathcal M(2;H,11)$ is reducible with atleast $2$ irreducible components (which is the first example of its kind). In a recent work of the present author with S.Pal, it's established that $\mathcal M(2;H,c_2)$ is irreducible for $c_2 \leq 10$ with the space being non-empty for $c_2 \geq 5, c_2 \neq 7$ (cf.\cite{DS}).\\

Let $\mathcal E \in \mathcal M(c_2)$. Then we have
\begin{equation}\label{e3}
 \mathcal E^* \otimes \mathcal O_X(5) \cong \mathcal E. 
\end{equation} 
 \eqref{e3} with Serre duality gives us the following :
\begin{align*}
H^2(X, \mathcal E) \cong & (H^0(X, \mathcal E^*\otimes \mathcal O_X(2)))^*\\
\cong & (H^0(X, \mathcal E\otimes \mathcal O_X(-3)))^*
\end{align*}
Since $\mathcal E$ is $\mu_H$ stable, it is $\mu_H$ semistable (see Lemma $1.2.13$, \cite{Huy}). Since any line bundle is  $\mu_H$ semistable (see \cite{Huy}) and tensor product of any two $\mu_H$ semistable bundles on a smooth projective variety is again $\mu_H$ semistable,  we have $\mathcal E(-3)$ is $\mu_H$ semistable. Let if possible,  $h^0(\mathcal E(-3)) \neq 0$. Then corresponding to a non-zero section of $H^0(\mathcal E(-3))$, one has an injective map $\mathcal O_X \hookrightarrow \mathcal E(-3)$. We have the corresponding slopes given by 
\begin{align*}
\mu_H(\mathcal O_X) = \frac{\text{deg}_H(\mathcal O_X)}{r(\mathcal O_X)}
 = \frac{c_1(\mathcal O_X).H}{1} = 0
\end{align*}
and 
\begin{align*}
\mu_H(\mathcal E(-3)) = \frac{\text{deg}_H(\mathcal E(-3))}{r(\mathcal E(-3))}
 = \frac{c_1(\mathcal E(-3)).H}{2} = -3.
\end{align*}
This is a contradiction to the $\mu_H$ semistability of $\mathcal E(-3)$, which forces $h^0(\mathcal E(-3)) = 0$ and hence $H^2(X,\mathcal E) =0$. Thus Theorem \ref{T4} allows us to talk about the Brill-Noether locus (as constructed in that theorem) in $\mathcal M(c_2)$.\\

\begin{remark}\label{r4.1}
Here we note that, instead if $\mathcal E$ is $\mu_H$ semistable vector bundle of rank $2$ on $X$ with $c_1(\mathcal E) = 5H$, $c_2(\mathcal E) =c_2$, then  $\mathcal E$ is  $\mu_H$ stable. Therefore, all four notions of semistability, stability coincide in this case.\\
\end{remark}

Let's consider the complete linear system $\mathcal L_5 :=|5H|$ (for definition and basic results on linear system see \cite{RH}). Let $C \in \mathcal L_5$ be a smooth, irreducible, projective, algebraic curve and $i : C \hookrightarrow X$ be the inclusion map.  From adjunction formula (see \cite{RH}, Chapter-$V$, Proposition $1.5$), we have the genus of $C$ :
\begin{center}
$g(C) = \frac{1}{2}C.(C + \omega_X) +1 = 106$\\
\end{center}

Now in the pursuit of studying the non-emptiness question we systematically arrange our investigation in somewhat reverse direction. In what follows we  first ask the following question : \textit{Assume the existence of a base-point free line bundle $L$  on the curve $C$ (as chosen in the previous paragraph) with $h^0(C,L) = 2$, then under what condition on it's degree $\text{deg}(L)$ we have $\mathcal W^1_2 \subset \mathcal M(c_2)$ is non-empty for $c_2 = \text{deg}(L)$ ?}\\

Since $L$ is base point free, it's globally generated and hence its pushforward $i_*L$ on $X$ is also globally generated. This means the canonical evaluation map $H^0(C,L) \otimes  \mathcal O_X \to i_*L$ is surjective, naturally we look at the following elementary transformation:
\begin{equation}\label{e5}
0 \to \mathcal K \to H^0(C,L) \otimes  \mathcal O_X \to i_*L \to 0
\end{equation}
Where, $\mathcal K$ is the kernel of the evaluation map. From Proposition \ref{P2} it follows that $\mathcal K$ is a vector bundle. 
Also note that, $\text{rank}(\mathcal K) = 2$. From computation of Chern classes using Proposition \ref{P2} and taking long exact cohomology sequence of \eqref{e5} and using cohomological facts listed in the preliminaries  we have the following :
\begin{itemize}
\item $c_1(\mathcal K) = -5H$.

\item $c_2(\mathcal K) = \text{deg}(L)$.

\item $H^0(X,\mathcal K) = H^1(X, \mathcal K) = 0$.
\end{itemize}

Applying the functor $\mathcal Ext^j(-,\mathcal O_X)$ to the exact sequence \eqref{e5} and noting the following  facts:\\ 
$(i)$ $\mathcal Hom(i_{*}L,\mathcal O_X) = 0$,\\
$(ii)$ $\mathcal Ext^1(\mathcal F, \mathcal O_X) =0, \text{ for all vector bundles}$ $\mathcal F$, we have :
\begin{equation}\label{e6}
0 \to \mathcal O_X^{\oplus 2} \to \mathcal K^{*} \to \mathcal Ext^{1}(i_{*}L,\mathcal O_X) \to 0
\end{equation}

Taking long exact cohomology sequence to the short exact sequence \eqref{e6}, 
  we obtain $h^0(X,\mathcal K^*) \geq 2$. Since  $c_1(\mathcal K^*) = - c_1(\mathcal K) =5H$ and  $c_2(\mathcal K^*) =  c_2(\mathcal K) = \text{deg}(L)$, if we can show that $\mathcal K^*$ is $\mu_H$ stable, then we have $\mathcal K^* \in \mathcal W^1_2 \subset \mathcal M(c_2)$, where $c_2 = \text{deg}(L)$. We know that if $\mathcal K$ is $\mu_H$ stable, then so is its dual. So, the problem reduces to the following :\textit{ Under what assumption on  $\text{deg}(L)$,  $\mathcal K$ is $\mu_H$ stable? }\\

Let if possible $\mathcal K$ is not $\mu_H$ stable. Then it's destabilized by a rank $1$ subsheaf $\mathcal G \subset \mathcal K$. If $\mathcal G$ destabilizes $\mathcal K$, then its double dual $\mathcal G^{**}$ also destabilizes $\mathcal K$ which is a line bundle and hence of the form $\mathcal O_X(s)$ for some integer $s$. Then we must have $\mu_H(\mathcal O_X(s)) \geq \mu_H(\mathcal K)$, which forces $s \geq -2$. Note that, from the short exact sequence \eqref{e5} we have $\mathcal O_X(s) \hookrightarrow  \mathcal O_X^{\oplus 2}$.  If $s >0$, then   $\mathcal O_X \hookrightarrow  \mathcal O_X(-s)^{\oplus 2}$, a contradiction (this can be seen by taking global section). This means $s \leq 0$. Since we have earlier obtained that $H^0(X, \mathcal K) =0$, $s$ can't be $0$. Thus we are left with the only possible values for $s$ as $-2,-1$.\\

Let's first consider the case, $s =-1$. In this case  rewriting $\mathcal Ext^1(i_*L, \mathcal O_X)$ as $O_C(C)\otimes L^*$ 
 and tensoring the short exact sequence  \eqref{e6} with $\mathcal O_X(-4)$ we have the following exact sequence :
\begin{equation}\label{e7}
0 \to \mathcal O_X(-4)^{\oplus 2} \to \mathcal K^* \otimes \mathcal O_X(-4) \to \mathcal O_C(H) \otimes L^* \to 0.
\end{equation}
If $\text{deg}(\mathcal O_C(H) \otimes L^*) = 30 - \text{deg}(L)<0$, then taking long exact cohomology sequence to \eqref{e7} one sees that, $h^0(\mathcal K^* \otimes \mathcal O_X(-4)) = 0$.\footnote{ This can be seen using cohomological equivalences mentioned in section \ref{S1} and the fact that if a line bundle  has negative degree then it has no non-trivial global section.} Using the isomorphism $\mathcal K^* \otimes \mathcal O_X(-4) \cong \mathcal K \otimes \mathcal O_X(1)$, we obtain $h^0( \mathcal K \otimes \mathcal O_X(1))=0$ and hence $\mathcal O_X(-1)$ can't be a line  subbundle of $\mathcal K$. This means under the assumption $\text{deg}(L) \geq 31$, the case $s = -1$ can't occur.\\

For $s=-2$ we carry out similar calculation. Here we tensor the short exact sequence  \eqref{e6} with $\mathcal O_X(-3)$ and obtain the following exact sequence :
\begin{equation}\label{e8}
0 \to \mathcal O_X(-3)^{\oplus 2} \to \mathcal K^* \otimes \mathcal O_X(-3) \to \mathcal O_C(2H) \otimes L^* \to 0.
\end{equation}
If $\text{deg}(\mathcal O_C(2H) \otimes L^*) =  60 - \text{deg}(L)<0$, then taking long exact cohomology sequence we see that $h^0(\mathcal K^* \otimes \mathcal O_X(-3)) = 0$. Using the isomorphism $\mathcal K^* \otimes \mathcal O_X(-3) \cong \mathcal K \otimes \mathcal O_X(2)$, we have $h^0( \mathcal K \otimes \mathcal O_X(2))=0$ and hence $\mathcal O_X(-2)$ can't be a line  subbundle of $\mathcal K$. This means under the assumption $\text{deg}(L) \geq 61$, the case $s= -2$ can't occur.\\
 
Together this implies if $\text{deg}(L) \geq \text{max}\{31,61\} =61$, then $\mathcal K$ is $\mu_H$ stable and hence $\mathcal K^*$ is $\mu_H$ stable. This forces $\mathcal K^* \in \mathcal W^1_2$. We summarize the above discussion in the following proposition.\\

\begin{proposition}\label{4.1}
Let's assume that there exists a base-point free line bundle $L$ on $C$   with $h^0(C,L)=2$ and  $\text{deg}(L) \geq 61$, then the Brill-Noether locus $\mathcal W^1_2 \subset \mathcal M(2;5H,c_2)$ is non-empty, provided $c_2 =\text{deg}(L)$.\\
\end{proposition}

Next we ask the question : \textit{for what values of  $d$, there exists a base-point free line bundle $L$ on $C$ such that $h^0(C,L) =2$  and $\text{deg}(L) =d$?}

We analyse what happens when $d \in \{103,104,105,106,107\}$. We basically mimic the calculation  done in \cite{KS}, Proposition $3.2$ with appropriate modification.\\

\textbf{(i) d  = 103} :

Since $g+r-d = 106 +1 -103 >0$, from Theorem \ref{T2.15} 
we have, $\text{dim}(W^1_{103}(C)) \geq 106 - 2.(106 -103 +1) = 106 -8 = 98$. From Lemma \ref{L2.16}, one sees that, $W^1_{103}(C) - W^2_{103}(C)$ is  non-empty. If there exists a base-point free line bundle in $W^1_{103}(C) - W^2_{103}(C)$, then we are done. If not, then any line bundle $L \in W^1_{103}(C) - W^2_{103}(C) $ is not base-point free. In this situation  tensoring each such line bundle 
 with the ideal sheaf of its base locus we obtain a family of base-point free line bundles with exactly $2$ sections (say, $\mathcal A$). We note that, $\text{dim}(\mathcal A) \geq 98$ and $\mathcal A \subset \cup_{e \leq 102}W^1_e(C)$. 
 Since $C \in \mathcal L_5$ is a smooth complete intersection of two smooth surfaces of degree $5$ and $6$ in $\mathbb P^3$, from
Proposition \ref{p22}, it follows that $C$ does not have  a $g^1_8$ and hence does not have a $g^1_4$. Then Theorem \ref{p20} forces $\text{dim}(\cup_{e \leq 102}W^1_e(C)) <  98$, a contradiction. This shows $W^1_{103}(C)$  contains a base-point free line bundle with exactly $2$ sections.\\

\textbf{(ii) d = 104} :

Again as before,  we have $\text{dim}(W^1_{104}(C)) \geq 106 - 2.(106 -104 +1) = 106 -6 = 100$ and $W^1_{104}(C) - W^2_{104}(C)$ is  non-empty. Again if $W^1_{104}(C) - W^2_{104}(C)$ does not contain any base- point free line bundle, then using the technique of tensoring with the ideal sheaf of the base locus, we get a family of base-point free line bundles (with exactly $2$ sections)  $\mathcal A$ such that $\text{dim}(\mathcal A) \geq 100$ and $\mathcal A \subset \cup_{2 \leq e \leq 103}W^1_e(C)$. It's not very difficult to see that $C$ is not hyperelliptic, trigonal or a smooth plane quintic. 
 Note that, $\omega_C \cong \mathcal O_C(7)$ and for the divisor $D = \mathcal O_C(6)$ we have an embedding  $C \hookrightarrow \mathbb P^3$  given by  the isomorphism $\omega_C \otimes \mathcal O_C(-D) \cong \mathcal O_C(1)$. Then the fact that $C$ is  not bi-elliptic follows from Proposition \ref{p26}.  From  Theorem \ref{p17}  followed by Theorem \ref{p18} we have, $\text{dim}(\cup_{2 \leq e \leq 103}W^1_e(C)) <  100$, a contradiction. This shows $W^1_{104}(C)$  contains a base-point free line bundle with exactly $2$ sections.\\

\textbf{(iii) d = 105} :

From Theorem \ref{p21},  there exist a base-point free complete pencil $g^1_{105}$ on $C$.\\

\textbf{(iv) d = 106} :

Again as before, we have $\text{dim}(W^1_{106}(C)) \geq 106 - 2.(106 -106 +1) = 106 -2 = 104$ and $W^1_{106}(C) - W^2_{106}(C)$ is  non-empty. If it does not contain any line bundle of desired kind, then by previously used technique of tensoring with the ideal sheaf of the base  locus, we obtain a family of base-point free line bundles with exactly $2$ sections $\mathcal A$ such that $\text{dim}(\mathcal A) \geq 104$ and $\mathcal A \subset \cup_{2 \leq e \leq 105}W^1_e(C)$.  From  Theorem \ref{p17} we have, $\text{dim}(\cup_{2 \leq e \leq 105}W^1_e(C)) \leq  102$, a contradiction. This shows $W^1_{106}(C)$  contains a base-point free line bundle with exactly $2$ sections.\\

\textbf{(v) d = 107} :

Since, $106 = \text{dim(Pic}^1(C)) \geq \text{dim}(W^1_{107}(C)) \geq 106 - 2.(106 -107 +1) = 106 $, $\text{dim}(W^1_{106}(C)) =106$. Again in the extreme unfavourable situation as before we use tensoring technique and obtain a family of base-point free line bundles with exactly $2$ sections $\mathcal A$ such that $\text{dim}(\mathcal A) = 106$ and is contained in $\cup_{ e \leq 106}W^1_d(C)$. Since $\text{dim}(W^0_e(C)) = e$ and $W^1_e(C)$ is a proper closed subset in $W^0_e(C)$, 
 this forces $\text{dim}(\cup_{ e \leq 106}W^1_e(C)) \leq 105 <106$, a contradiction. This shows $W^1_{106}(C)$  contains a base-point free line bundle with exactly $2$ sections.\\

We conclude the content of the above discussion in the following proposition.\\

\begin{proposition}\label{4.2}
If $d \in \{103,104,105,106,107\}$, then for each such $d$ there exists a base-point free line bundle $L$ on $C$ with $h^0(C,L)=2$ such that $\text{deg}(L) =d$.\\
\end{proposition}

We are now in a position to prove the first part of the promised theorem \ref{PP1}.\\

\begin{theorem}\label{T4.3}
If $103 \leq c_2 \leq 107$, then $\mathcal W^{1}_{2} - \mathcal W^{2}_{2} \subset \mathcal M(2;5H,c_2)$ is non-empty.
\end{theorem}

\begin{proof}
For each $c_2 \in \{103,104,105,106,107\}$, Proposition \ref{4.2} says that there exists a base-point free line bundle $L$ of $\text{deg}(L) =c_2$ on $C$ with exactly $2$ sections. Since all of these numbers mentioned above is $\geq 61$, Proposition \ref{4.1} applies and we have $\mathcal W^1_2 \subset \mathcal M(2;5H,c_2)$ is non-empty. Now the only thing left to prove is that the element chosen  for showing the non-emptiness of $\mathcal W^1_2$ in proposition \ref{4.1} can be constructed in such a way that it has exactly $2$ sections. But this can be done because the space of base-point free line bundles $L$ on $C$ with exactly $2$ independent global sections and $h^0(X, \mathcal Ext^1(i_*L,\mathcal O_X)) =0$ is non-empty (for the mentioned values of $c_2$). If we choose line bundle from this space, then we are through.
\end{proof}

We end this section by pointing out the following remarks.\\

\begin{remark}\label{R4.5}
$(i)$ We observe that for every pair of  numbers $(d,n)$ satisfying $5 \leq d < 6+n$, one has for every $\mathcal E \in \mathcal M(2;(d-2+n)H,c_2)$ on a very general surface $S$ of degree $d$ in $\mathbb P^3$, $H^2(S,\mathcal E) =0$. In particular if $t \geq 5$, then for any $\mathcal E \in \mathcal M(2;tH,c_2)$ on $X$, one has $H^2(X,\mathcal E) =0$ and therefore it makes sense to talk about their Brill-Noether locus by theorem \ref{T4}. We have only considered here the case $t=5$. \\

$(ii)$ If we write $t = 4+n$, where $n \in \mathbb N$, then similar results can be obtained for any such $t$. To be more precise we obtain the following results :\\
\begin{itemize}
\item Let's assume that there exists a base-point free line bundle $L$ on a smooth, irreducible, projective, algebraic curve $C \in |(4+n)H|$ with $\text{deg}(L) \geq 6(4+n)\lfloor{2+ \frac{n}{2}}\rfloor +1$, then the Brill-Noether locus $\mathcal W^1_2 \subset \mathcal M(2;(4+n)H,c_2)$ is non-empty, provided $c_2 =\text{deg}(L)$. Note that in this case genus of such a curve is given by $g = 3(4+n)(6+n) +1$.\\

\item If $d \in \{g-3,g-2,g-1,g,g+1\}$, then for each such $d$ there exists a base-point free line bundle $L$  on $C$ with $h^0(C,L) =2$ such that $\text{deg}(L) =d$.\\

\item  If $g-3 \leq c_2 \leq g+1$, then $\mathcal W^{1}_{2} - \mathcal W^{2}_{2} \subset \mathcal M(2;tH,c_2)$ is non-empty.\\
\end{itemize}

$(iii)$   We have for $103 \leq c_2 \leq 107$, the generalized Brill-Noether number $\rho^1_2 \geq 2c_2 -53 >0 $ and corresponding Brill-Noether loci are also non-empty.
\end{remark}

\section{Existence of smooth points in $\mathcal W^1_2$ and irreducible components having expected dimension}\label{S4}
    
Here we first establish the existence of smooth points in $\mathcal M(2;5H,c_2)$.\\

\begin{proposition}\label{5.1}
Let $103 \leq c_2 \leq 107$ and $\mathcal E \in \mathcal W^1_2 - \mathcal W^2_2 \subset \mathcal M(2;5H,c_2)$ be a general member constructed as  in Proposition \ref{4.1} for establishing non-emptiness, then $\mathcal E$ is a smooth point of $\mathcal M(2;5H,c_2)$.
\end{proposition}

\begin{proof}
Here the main idea is to move from the moduli space at hand to some moduli space about which a number of geometric results are known. Indeed we see that the map given by
\begin{align*}
f : & \mathcal M(2;5H,c_2) \to \mathcal M(2;H,c_2-36)\\
& \mathcal E \mapsto \mathcal E \otimes \mathcal O_X(-2)
\end{align*}
is an isomorphism\footnote{ It's easy to see the following :
\begin{itemize}
\item if $\mathcal E$ is a rank $2$, $\mu_H$-stable vector bundle, then so is $\mathcal E \otimes \mathcal O_X(-2)$.

\item if $c_1(\mathcal E) = 5H, c_2(\mathcal E) = c_2$, then using results from section \ref{S1} we have $c_1(\mathcal E \otimes \mathcal O_X(-2)) = 5H-4H =H, c_2(\mathcal E \otimes \mathcal O_X(-2)) = c_2 -36$.

\item the map in the opposite direction is given by $\mathcal F \mapsto \mathcal F \otimes \mathcal O_X(2)$, which makes $f$ an isomorphism. 

\end{itemize} }.This means $\mathcal E$ is smooth in $\mathcal M(2;5H,c_2)$ iff $\mathcal E \otimes \mathcal O_X(-2)$  is smooth in $\mathcal M(2;H,c_2-36)$. This enables us to concentrate  proving the smoothness of $\mathcal G := \mathcal E \otimes \mathcal O_X(-2)$ in $\mathcal M(2;H,c_2-36)$. We give a proof by contradiction.\

Let if possible $\mathcal G$ is not a smooth point of $\mathcal M(2;H,c_2-36)$.  Since it's not smooth, it's \textit{potentially obstructed} (see discussion before Proposition $5.1$, \cite{Simpson11}), i.e $h^2(X,\text{End}^0(\mathcal G)) >0$. By Serre duality, we have, $h^0(X, \text{End}^0(\mathcal G) \otimes \omega_X) >0$\footnote{This is because the trace free part is self dual, i.e $\text{End}^0(\mathcal G)^* \cong\text{End}^0(\mathcal G)$ }. Since the obstruction space of $\mathcal G$ is non-zero, from Lemma $2.1$ of \cite{Simpson11}, we have there exists a non-trivial Higgs field $\phi : \mathcal G \to \mathcal G \otimes \omega_X$ with trace $\text{tr}(\phi) =0$. From Proposition \ref{p8} one sees that it's a potentially obstructed bundle of type $(ii)$ as described in that proposition (Note that it can't be of type $(i)$ by assumption). Now from Proposition \ref{p9}, the dimension of locus of such potentially obstructed bundles is $\leq 39$. But the dimension of family of base-point free line bundles $L$ on $C$ such that $103 \leq \text{deg}(L) \leq 107$, $h^0(C,L) =2$, $h^0(X, \mathcal Ext^1(i_*L,\mathcal O_X)) =0$ has dimension atleast $40$, a contradiction (from the discussions in the previous section \ref{S3}). This means $\mathcal G$ is  a smooth point of $\mathcal M(2;H,c_2-36)$ and consequently $\mathcal E$ is smooth in $\mathcal M(2;5H,c_2)$. \\
\end{proof}
 
 
Next we prove that the same smooth point of $\mathcal M(2;5H,c_2)$ is also a smooth point of $\mathcal W^1_2$ using the injectivity of Petri map. To be more precise we prove the promised Theorem \ref{PP2}. The same proof as that of Theorem $3.5$, \cite{KS} works. For convenience we present the proof in our case.\\

\begin{theorem}
If $103 \leq c_2 \leq 107$, then  $\mathcal W^{1}_{2} \subset \mathcal M(2;5H,c_2)$ contains a smooth point and  the irreducible component containing it has the actual dimension same as the expected dimension.
\end{theorem}

\begin{proof}
Let $\mathcal E \in \mathcal W^{1}_{2} - \mathcal W^{2}_{2}$ be a smooth point in $\mathcal M(2,5H,c_2)$ as described in the previous section \ref{S3} \footnote{Existence of such an $\mathcal E$ is guaranteed by proposition \ref{5.1}}.  Then from the short exact sequence \eqref{e6} one sees that $\mathcal E$ fits into the following  short exact sequence :
\begin{equation}\label{e5.1}
0 \to \mathcal O_X^{\oplus 2} \to \mathcal E \to \mathcal O_C(C)\otimes L^* \to 0 
\end{equation}

Tensoring the  exact sequence \eqref{e5.1} by $\mathcal E^* \otimes \omega_X$ and noting that $\mathcal E \otimes \mathcal E^* \cong \text{End}(\mathcal E)$ we have the following exact sequence :
\begin{equation}\label{e5.2}
0 \to (\mathcal E^* \otimes \omega_X)^{\oplus 2} \to  \text{End}(\mathcal E) \otimes \omega_X \to \mathcal O_C(C)\otimes L^* \otimes (\mathcal E^* \otimes \omega_X)|_C \to 0
\end{equation}

From Serre duality we have, $h^0(X,\mathcal E^* \otimes \omega_X) = h^2(X,\mathcal E) =0$  and as $\mathcal E$ is a smooth point in the moduli space $\mathcal M(2;5H,c_2)$ one has, $H^0(X, \text{End}(\mathcal E) \otimes \omega_X) \cong H^0(X, \omega_X)$. Thus by taking cohomology long exact sequence  to short exact sequence \eqref{e5.2}, we get the following exact sequence :
\begin{align*}\label{5.3}
0 \to H^0(X, \omega_X) \to & H^0(C, \mathcal O_C(C)\otimes L^* \otimes (\mathcal E^* \otimes \omega_X)|_C)\\
 \to & H^1(X,(\mathcal E^* \otimes \omega_X)^{\oplus 2}) \xrightarrow{\theta} H^1(X, \text{End}(\mathcal E) \otimes \omega_X) \to \cdots .
\end{align*}

We observe that the  map $\theta$ in the above short exact sequence is same as the Petri map $\mu$ defined in \eqref{Petri}.  Thus in order to show that $\mathcal E$ is a smooth point of $\mathcal W^{1}_{2}$, it's enough to show $\theta$ is injective. We dualize the short exact sequence \eqref{e5.1} and restrict it to the curve $C$ to obtain the following exact sequence (on $C$) :
\begin{equation}\label{e5.4}
0 \to \mathcal O_C(-C) \otimes L \to \mathcal E^{*}|_C \to \mathcal O^{\oplus 2}_C \to L\to 0
\end{equation}
 
which induces the another short exact sequence on $C$ given by:
\begin{equation}\label{e5.5}
0 \to \mathcal O_C(-C) \otimes L \to \mathcal E^{*}|_C \to L^* \to 0
\end{equation}

Tensoring \eqref{e5.5} by $\mathcal O_C(C) \otimes L^{*} \otimes \omega_{X}|_C$ and taking long exact cohomology sequence we have :
\begin{align*}
0 & \to H^0(C,\omega_X|_C ) \to H^0(C,\mathcal O_C(C) \otimes L^* \otimes (\mathcal E^* \otimes \omega_X)|_C )\\
& \to H^0(C,\mathcal O_C(C) \otimes L^{*^{\otimes 2}} \otimes \omega_X|_C ) \to \cdots
\end{align*}
From our assumption on $L$, $H^0(C,\mathcal O_C(C) \otimes L^*) = 0$ and as $\text{deg}(L^* \otimes \omega_X|_C ) < 0$, we have $H^0(C,\mathcal O_C(C) \otimes L^{*^{\otimes 2}} \otimes \omega_X|_C ) = 0$. Thus from the above exact sequence one obtains, $H^0(C,\omega_X|_C ) \cong H^0(C,\mathcal O_C(C) \otimes L^{*} \otimes (\mathcal E^{*} \otimes \omega_X)|_C )$. This along with the observation  $H^0(X,\omega_X)\cong H^0(C,\omega_X|_C )$ forces  the map $\theta$ to be injective.\\
\end{proof}

We end this section with the following remark.\\

\begin{remark}
Let $t\geq 5$ be such that $t$ is odd. Let's rewrite $t = 4+n$ as before. In this setting using the same techniques as discussed in this section we can obtain the following general  result : If $3(4+n)(6+n)-2 \leq c_2 \leq 3(4+n)(6+n)+2$, then $\mathcal W^1_2 \subset \mathcal M(2;tH,c_2)$ contains a smooth point and the irreducible component containing it has the actual dimension same as the expected dimension.
\end{remark}

 \section{Non-emptiness of Brill-Noether locus $\mathcal W^0_2$ and components  having expected dimension}\label{S5}
 
In this section our intention is to prove the promised Theorem \ref{PP3}. We begin this section by giving a lower bound on $c_2$ such that $\mathcal W^0_2 \subset \mathcal M(2;5H,c_2)$ becomes non-empty.\\

By $\text{Hilb}^{c_2}(X)$  we mean the Hilbert scheme parametrizing all zero-dimensional subschemes  of $X$ of length $c_2$ and by $\text{Hilb}^{c_2}(X)^{\text{l.c.i}}$ be the open subscheme of the Hilbert scheme $\text{Hilb}^{c_2}(X)$ consisting of length $c_2$ subschemes of $X$ which are locally complete intersections.  For any $Z \in \text{Hilb}^{c_2}(X)^{\text{l.c.i}}$, we have an exact sequence of the following form :
\begin{equation}
0 \to \mathcal O_X \to \mathcal F \to \mathcal J_{Z} \otimes \mathcal O_X(5) \to 0
\end{equation}

where $\mathcal F$ is a rank $2$ torsion free sheaf on $X$.\\

\textbf{Claim }: for $c_2 \geq 117$, we have $\mathcal F \in \mathcal W^0_2$.\\

Since $\mathcal O_X \hookrightarrow \mathcal F$, we already have $h^0(\mathcal F) \geq 1$. Assume that we are able to show that $\mathcal F$ is a vector bundle, then it can also be shown that $\mathcal F$ is $\mu_H$ stable. This enables us to concentrate proving $\mathcal F$ is a vector bundle, for $c_2 \geq 117$. In order to show that $\mathcal F$ is a vector bundle we proceed as follows: Note that, the space of such isomorphism classes of extensions  is parametrized by $\mathbb P(\text{Ext}^{1}(\mathcal J_{Z} \otimes \mathcal O_X(5), \mathcal O_X))$. By Serre duality, we have $\text{Ext}^{1}(\mathcal J_{Z} \otimes \mathcal O_X(5), \mathcal O_X) \cong (H^{1}(X,\mathcal J_{Z} \otimes \mathcal O_X(7)))^*$.  We see that for a general element  $ Z $ with $l(Z) =c_2 \geq 117$, one has $h^0(\mathcal J_{Z} \otimes \mathcal O_X(7))=0$ (because, $h^0(X, \mathcal O_X(7)) = 116 $). Therefore, a general element of $\text{Hilb}^{c_2}(X)^{\text{l.c.i}}, c_2 \geq 117$ satisfies $\text{CB}(7)$ and hence from Theorem \ref{CB}, one sees that  $\mathcal F$ is a rank $2$ vector bundle.\\

\begin{remark}\label{r6.1}
Along the same lines of Theorem $2.8$, \cite{S} it can be shown that $\mathcal M(c_2)$ as introduced in the beginning of  section \ref{S3} is good for $c_2 \geq 117$. Therefore, we have the generalized Brill-Noether number $\rho^0_2 =3c_2 -117 >0$ for $c_2 \geq 117$ and the corresponding Brill-Noether locus $\mathcal W^0_2$ is non-empty.\\
\end{remark}

The following two results give a bound for the dimension of $\mathcal W^{0}_{2}$. \\

\begin{lemma}
If $c_2 \geq 117$, we have $\text{dim}(\mathcal W^{0}_{2}) \leq 3c_2 -117$.
\end{lemma}

\begin{proof}
This proof closely follows \cite{S}, Proposition $2.3$. Let $\mathcal E \in \mathcal W^0_2$. Since $h^0(\mathcal E) >0$, there is a section $s$ corresponding to a morphism $s : \mathcal O_X \to \mathcal E$. The zeros of $s$ are in codimension $2$. Therefore, $s$ fits into an exact sequence of the form :
\begin{equation}
0 \to \mathcal O_X \to \mathcal E \to \mathcal J_{Z} \otimes \mathcal O_X(5) \to 0
\end{equation}

where $Z$ is a zero dimensional locally complete intersection subscheme with $l(Z) = c_2$ satisfying $\text{CB}(7)$. Let $\mathcal N $ be the space of pairs :
\begin{center}

$\mathcal N = \{(\mathcal E,t) | \mathcal E \in \mathcal W^0_2, t \in \mathbb P(H^0(X,\mathcal E))\} $\\
\end{center}

Let's consider the obvious maps $p_1 : \mathcal N \to \mathcal W^0_2$ and $p_2 :  \mathcal N \to \text{Hilb}^{c_2}(X)$. Since $p_1$ is surjective, we have $\text{dim}(\mathcal W^{0}_{2}) \leq \text{dim}(\mathcal N)$. Note that, $\text{dim}(p_2^{-1}(Z)) = \text{dim}(\mathbb P(\text{Ext}^{1}(\mathcal J_{Z} \otimes \mathcal O_X(5), \mathcal O_X))) = h^{1}(X,\mathcal J_{Z} \otimes \mathcal O_X(7)) -1 $. From the canonical exact sequence :
\begin{equation}
0 \to \mathcal J_{Z} \otimes \mathcal O_X(7) \to \mathcal O_X(7) \to \mathcal O_Z(7) \to 0
\end{equation}
we obtain that $h^1(\mathcal J_{Z} \otimes \mathcal O_X(7)) = c_2 -116$, for a general element $ Z $ with $l(Z) =c_2 \geq 117$. Since $\text{dim}(\text{Hilb}^{c_2}(X)) = 2c_2$, one has $\text{dim}(\mathcal N) \leq 3c_2 -117$. Note that this dimension estimate is over general points.\\

Let's consider other subsets, $\Delta_i := \{ Z \in \text{Hilb}^{c_2}(X)| h^0(\mathcal J_{Z} \otimes \mathcal O_X(7)) \geq i\} $. Consider the incidence variety $T = \{(C,Z) | Z \subset C\} \subset \mathbb P(H^0(\mathcal O_X(7))) \times \text{Hilb}^{c_2}(X) $ and let $\pi_1,\pi_2$ be projections on first and second factors. In this situation we have $\text{dim}(\pi_2^{-1}(C)) \leq c_2$ and hence $115 + c_2 \geq \text{dim}(T) \geq \text{dim}(\pi_2)^{-1}(\Delta_i) \geq \text{dim}(\Delta_i) +i -1$. This forces $\text{dim}(\Delta_i)$ is bounded by $c_2 + 116 -i \leq 2c_2 -i$ as $c_2 \geq 117$. Therefore, the codimension of $\Delta_i$ in $\text{Hilb}^{c_2}(X) $ is $\geq i$. This implies $\text{dim}(\mathcal N) \leq 3c_2 -117$.

\end{proof}

\begin{lemma}[\cite{Simpson11}, Corollary $3.1$]
Every irreducible component of $\mathcal W^{0}_{2}$ has dimension $\geq 3c_2 - h^0(X, \mathcal O_X(5)\otimes \omega_X)-1$.\\
\end{lemma}

Since, $h^0(X, \mathcal O_X(5)\otimes \omega_X) =116$, we have for $c_2 \geq 117$ every irreducible component of $\mathcal W^{0}_{2}$ has dimension exactly $3c_2 -117$. On the other hand as we have already seen in remark \ref{r6.1} that the Brill-Noether number is $\rho^{0}_{2} = 3c_2 -117$.\\

We summarize the above discussion as the following theorem :\\

\begin{theorem}
If $c_2 \geq 117$, then the Brill-Noether locus $\mathcal W^{0}_{2} \subset \mathcal M(2;5H.c_2)$ is non-empty and every irreducible component of it has the actual dimension same as expected dimension.\\
\end{theorem}

We end this section with the following remark.\\

\begin{remark}
Similar calculations (with minor changes) as that of Lemma $5.1$, \cite{S} enables us to obtain a lower bound for $c_2$ such that $\overline{\mathcal W^0_2} \subset \overline{\mathcal M(c_2)}$ is connected.
\end{remark}

\section{Ulrich bundles on $X$}\label{S6}

This final section is divided into two subsections. In the first subsection we study the possibility of existence of rank $2$ weakly Ulrich bundles on $X$ and in the second subsection we study the implications of the existence of rank $2$ Ulrich bundles on $X$.\\

\subsection{Existence of weakly Ulrich bundles and non-emptiness of certain Brill-Noether Loci}
We begin this section with a Proposition which gives a list of sufficient  conditions on a line bundle $L$ on $C \in \mathcal L_5$ of degree atleast $91$  for the existence of a $\mu_H$ stable, rank  $2$ weakly Ulrich bundle on $X$. The following treatment closely follows \cite{Coskunconstruction}.\\

\begin{proposition}\label{P7}
Let $C \in \mathcal L_5$ be a smooth irreducible projective algebraic curve as described in section \ref{S3}. Let $L$ be a line bundle  of degree $\geq 91$ on $C$ satisfying the following properties :\\

$(i)$ $L$ is base-point free, i.e globally generated,\

$(ii)$ $h^0(L) = 2$,\

$(iii)$ $h^0(L(-1)) = 0$,\

$(iv)$ $h^0((L^*)^2(5)) =0 $,

$(v)$ The maps $H^0(L) \otimes H^0(\mathcal O_X(m))  \to H^0(L(m))$ for $m=1,2$ is surjective.

Then the rank $2$ vector bundle $\mathcal E$ constructed from the sequence
\begin{equation}\label{e7.1}
0 \to \mathcal E^* \to H^0(L)\otimes \mathcal O_X \to i_*L \to 0 
\end{equation}

satisfies the following properties :

$(a)$ $c_1(\mathcal E) = 5H$, $c_2(\mathcal E) = \text{deg}(L)$,\

$(b)$ $\mathcal E$ is weakly Ulrich,\

$(c)$ $\mathcal E$ is globally generated in codimension $1$,\

$(d)$ $\mathcal E$ is simple and\

$(e)$ $\mathcal E$ is $\mu_H$ stable.

\end{proposition}

\begin{proof}
From the Chern class computations on the short exact sequence \eqref{e7.1} (using Proposition \ref{P2}) we have $c_1(\mathcal E) = 5H$ and $c_2(\mathcal E) = \text{deg}(L)$ and hence the condition $(a)$ is satisfied.\

Since $\mathcal E$ is of rank $2$ and determinant $\mathcal O_X(5)$, we have $\mathcal E^*(m+5) \cong \mathcal E(m), \forall m \in \mathbb Z$. From Serre duality we obtain:
\begin{center}
$H^i(\mathcal E(m)) \cong H^i(\mathcal E^*(m+5)) \cong H^{2-i}(\mathcal E(-m-3))^*$
\end{center}
Therefore in order to show that $\mathcal E$ is weakly Ulrich it's sufficient to show the vanishing of $H^0(\mathcal E(m))$ for all $m \leq -2$ and of $H^1(\mathcal E(m))$ for all $m \geq 0$.
Dualizing the short exact sequence \eqref{e7.1} we obtain 
\begin{equation}\label{e7.2}
0 \to H^0(L)^* \otimes \mathcal O_X \to \mathcal E \to \mathcal O_C(C) \otimes L^* \to 0.
\end{equation}
Twisting the short exact sequence \eqref{e7.2} by $\mathcal O_X(-2)$ and taking long exact cohomology sequence we obtain $H^0(\mathcal E(-2)) \cong H^0(L^*(3))$. Since $\text{deg}(L^*(3)) = 90 -\text{deg}(L) <0$, we have $H^0(L^*(3))=0$ and hence $H^0(\mathcal E(-2)) =0$. It can be easily seen that the same technique yields $H^0(\mathcal E(m))=0 , \forall m <-2$.\

Twisting the short exact sequence \eqref{e7.1} by $\mathcal O_X(-m+2)$ and taking global sections we obtain the following :
\begin{align*}
& 0 \to H^0(\mathcal E^*(-m+2)) \to H^0(L)\otimes H^0(\mathcal O_X(-m+2) \to H^0(L(-m+2)) \\
& \to H^1(\mathcal E^*(-m+2)) \to 0.
\end{align*}
Noting that $H^0(\mathcal E^*(-m+2)) =0$ for $m \geq -1$ and $H^1(\mathcal E^*(-m+2)) \cong H^1(\mathcal E(m))^*$ we obtain the following short exact sequence for $m \geq -1$:
\begin{align*}
0 \to H^0(L) \otimes H^0(\mathcal O_X(-m+2)) \to H^0(L(-m+2)) \to H^1(\mathcal E(m))^* \to 0.
\end{align*}
This means $H^1(\mathcal E(m))^* \cong \text{coker}(H^0(L) \otimes H^0(\mathcal O_X(-m+2)) \to H^0(L(-m+2)))$. We see that from assumption $(v)$ we have for $m=0,1,2$ the map  $H^0(L) \otimes H^0(\mathcal O_X(-m+2)) \to H^0(L(-m+2)))$ is an isomorphism and for those values of $m$ we have $H^1(\mathcal E(m))=0$. For $m >2$ we have from the assumption $(iii)$ that $H^0(L(-t)) =0, \forall t >0$ and hence $H^1(\mathcal E(m))=0, \forall m >2$. Therefore the condition $(b)$ is satisfied.\

From the short exact sequence \eqref{e7.2} noting that $H^0(L)^* \hookrightarrow H^0(\mathcal E)$, we have $\mathcal E$ is globally generated away from the base locus of $\mathcal O_C(C) \otimes L^*$, i.e  it is globally generated in codimension $1$. Hence condition $(c)$ is satisfied.\

Tensoring the short exact sequence \eqref{e7.2} with $\mathcal E^*$ and taking long exact cohomology sequence we obtain the isomorphism $H^0(\mathcal E \otimes \mathcal E^*) \cong H^0(\mathcal E^*|_C \otimes \mathcal O_C(C) \otimes L^*)$. Now restricting the short exact sequence \eqref{e7.1} to the curve $C$ we have the following short exact sequence :
\begin{equation}\label{e7.3}
0 \to \mathcal O_C(-C) \otimes L \to \mathcal E^{*}|_C \to  H^0(L) \otimes \mathcal O_C \to L\to 0
\end{equation}
 
which induces the another short exact sequence on $C$ given by:
\begin{equation}\label{e7.4}
0 \to \mathcal O_C(-C) \otimes L \to \mathcal E^{*}|_C \to L^* \to 0
\end{equation}
Tensoring \eqref{e7.4} by $\mathcal O_C(C) \otimes L^{*}$ and taking long exact cohomology sequence we have :
\begin{align*}
0 & \to H^0(\mathcal O_C ) \to H^0(\mathcal E^*|_C \otimes \mathcal O_C(C) \otimes L^*) \to H^0((L^{*})^{2}(5)) \to 0
\end{align*}
From assumption $(iv)$, we have the following isomorphism $H^0(\mathcal O_C ) \cong H^0(\mathcal E^*|_C \otimes \mathcal O_C(C) \otimes L^*)$. This shows $\mathcal E$ is simple and hence the condition $(d)$ is satisfied.\

In order to show $\mathcal E$ is $\mu_H$ stable it's enough to show that $\mathcal E^*$ is so. Let if possible $\mathcal E^*$ is not $\mu_H$ stable, then arguing as in section \ref{S3}, one can show that it can be only destabilized by either $\mathcal O_X(-2)$ or $\mathcal O_X(-1)$, a contradiction as $H^0(\mathcal E^*\otimes \mathcal O_X(2)) \cong H^2(\mathcal E)^* =0$ and $H^0(\mathcal E^* \otimes \mathcal O_X(1)) \cong H^2(\mathcal E \otimes \mathcal O_X(1))^* =0$. Hence condition $(e)$ is satisfied.\\

\end{proof}

\begin{remark}\label{R7.2}
Here we mention that even though property $(d)$ follows from property $(e)$ and therefore condition $(iv)$ is redundant, we have given an independent proof for simplicity.\\
\end{remark}

Next we make an attempt to investigate whether the conditions $(i)$, $(ii)$, $(iii)$ and $(v)$ of Proposition \ref{P7} really hold.\\

By Theorem \ref{T2.15}, we have for $ 91 \leq d \leq 107$ each irreducible component of $W^1_d(C)$ has dimension atleast $\rho(g,r,d) \geq 74$. From Lemma \ref{L2.16}, we obtain that each such component has a non-empty Zariski open subset whose members satisfy condition $(ii)$ of the previous Proposition.\

Note that, if $91 \leq d \leq 107$, then for all $L \in \text{Pic}^d(C)$, one has $61 \leq \text{deg}(L(-1)) \leq 77$. Since $\text{dim}(W^0_e(C)) =e$, the general member  of $W^1_d(C)$ satisfies condition $(iii)$.\

We know that the image of the addition map $C \times W^1_{d-1}(C) \to W^1_d(C)$ given by $(p, L) \mapsto L(p)$, parametrizes all the members of $W^1_d(C)$ which are not base point free. Therefore, if the dimension of  the image of the above map is atmost $\rho(g,r,d)-1$, then the general member  of $W^1_d(C)$ satisfies condition $(i)$ of Proposition \ref{P7}. From cohomological facts listed in the preliminaries, we have the restriction map $H^0(\mathcal O_{\mathbb P^3}(5)) \to H^0(\mathcal O_X(5))$ is an isomorphism and hence $C$ can be thought of as $X \cap Y$, where $Y\subset \mathbb P^3$ is a smooth quintic hypersurface. Therefore, we have the information regarding maximum number of collinear points in $C$ (say, $l$). At this point we note down a theorem which gives us gonality and Clifford index of the curve $C$ using $l$.\\

\begin{theorem}\label{T7.3}
Let $C$ be a smooth, nondegenerate complete intersection curve in $\mathbb P^3$. Let $l$ be the maximum number of collinear points on $C$. Then :\\
$(i)$ $\text{gon}(C) = \text{deg}(C) -l$, and an effective divisor $\Gamma \subset C$ computes this gonality if and only if $\Gamma$ is residual, in a plane section of $C$, to a set of $l$ collinear points of $C$.\\
$(ii)$ If $\text{deg}(C) \neq 9$, then $\text{Cliff}(C) = \text{gon}(C) -2$.
\end{theorem}

\begin{proof}
For part $(i)$ see \cite{Basili}, Th\'{e}or\`eme $4.2$. For part $(ii)$ see \cite{Basili}, Th\'{e}or\`eme $4.3$.\\
\end{proof}

Having these numerical invariants of the complete intersection curve $C$ at hand,  we now mention another interesting thoerem on $K3$ surface which makes use of the numerical inavriants obtained from the Theorem \ref{T7.3} to give a useful dimension estimate result in this context. Note that, the following theorem is on $k3$ surface and the curve mentioned in the theorem is a general curve.\\

\begin{theorem}
Let $C$ be a smooth curve on a $K3$ surface with $\text{gon}(C) =k$, $\text{Cliff}(C) =k-2$ and $\rho(g,1,k) \leq 0$ such that the linear system $|C|$ is base-point free. If $d \leq g-k+2$, then every irreducible component of $W^1_d(C)$ has dimension atmost $d-k$.
\end{theorem}

\begin{proof}
See \cite{AF}, Theorem $3.2$.\\
\end{proof}

 It would be interesting to obtain an analogous result for general curves on  very general surfaces of degree $d \geq 5$ in $\mathbb P^3$ (We do not have such an analogous result at the moment, we are also unable to find such a result in the literature). Anyway if by such an analogous result or otherwise we are able to get a dimension estimate of $W^1_{d-1}(C)$ (and hence a dimension estimate of the image of the addition map), then there is a possibility that for some specified values of $d$ (between $91$ and $102$) we will be able to say that the general member  of $W^1_d(C)$ satisfies condition $(i)$ of Proposition \ref{P7}.\

Finally we mention that, if a line bundle $L$ satisfies all the conditions $(i),(ii),(iii)$, then condition $(v)$ is equivalent to the injectivity of $H^1(F(m)) \to H^1(L) \otimes H^1(\mathcal O_C(m))$, for $m=1,2$, where $F$ is the kernel of the evaluation map for $L$.\\

We summarize the above discussion in the following remarks.\\

\begin{remark}
$(i)$ Let  $C \in \mathcal L_5$ be a general smooth irreducible projective algebraic curve as mentioned in section \ref{S3}. Then for  $91 \leq d \leq 107$, there exists a line bundle of degree $d$ on $C$ (the general member of $W^1_d(C)$) satisfying conditions $(ii),(iii)$ of Proposition \ref{P7}. Moreover, If for some specified $d$ in that range  the general member of $W^1_d(C)$ is base-point free, then for those values of $d$ we have obtained a line bundle $L$ of degree $d$ on $C$ satisfying conditions $(i),(ii),(iii)$.  Finally, If for such an $L$ the maps $H^1(F(m)) \to H^1(L) \otimes H^1(\mathcal O_C(m))$, for $m=1,2$ as discussed before are injective, then for those values of $d$ we have found a line bundle $L$ of degree $d$ satisfying all the desired conditions (i.e conditions $(i),(ii),(iii),(v)$) of Proposition \ref{P7}.\\

$(ii)$ If there exists an weakly Ulrich bundle $\mathcal E$ of rank $2$ on $X$ such that $ c_1(\mathcal E) = 5H$ and $91 \leq c_2(\mathcal E) \leq 107$, then by  Remark \ref{r2.31} it's not an Ulrich bundle and therefore serves as an example of weakly Ulrich bundle that is not Ulrich.\\

$(iii)$ If there exists an weakly Ulrich bundle on $X$ which is constructed as in Proposition \ref{P7} (i.e from short exact sequence \eqref{e7.1}), then the corresponding Brill-Noether locus $\mathcal W^1_2$ is non-empty.  In Theorem \ref{T4.3} we have already shown that if $103 \leq c_2 \leq 107$, then $\mathcal W^{1}_{2} - \mathcal W^{2}_{2} \subset \mathcal M(2;5H,c_2)$ is non-empty. Now If we can show that for some specified values of $d$ with $91 \leq d \leq 102$, there exists a line bundle $L$ (on $C$) of degree $d$ satisfying conditions $(i),(ii)$ of proposition \ref{P7}, then we have for those values of $d$ the non-emptiness of the Brill-Noether locus $\mathcal W^1_2 \subset \mathcal M(2;5H,c_2)$ with $c_2 =d$.\\

\end{remark}

\subsection{Implications of the Existence of a rank $2$ Ulrich bundle on $X$}\

In this final subsection we explore what the existence of a rank $2$ Ulrich bundle  on $X$ means from the point of view of the non-emptiness of Brill-Noether locus and the existence of higher rank Ulrich bundles on $X$. We begin this section by recalling the results which give us the existence and stability of a rank $2$  Ulrich bundles on $X$.\\

\begin{proposition}\label{p7.7}
There exists a rank $2$ Ulrich bundle on $X$.
\end{proposition}

\begin{proof}
Follows from Lemma \ref{L2.33} and Proposition \ref{P2.32}.\\
\end{proof}

\begin{proposition}\label{p7.8}
On  $X$, there does not exist any Ulrich bundle of odd rank and hence any rank $2$ Ulrich bundle on $X$ is stable.
\end{proposition}
 
\begin{proof}
Let $\mathcal E$ be an Ulrich bundle of odd rank $r$ on $X$. Then from Proposition \ref{p2.30} $(ii)$ we see that it has degree $15r$ which must be divisible by the degree of the hyperplane class i.e $6$, a contradiction. The stability of a rank $2$ Ulrich bundle follows from Proposition \ref{p2.34} $(i)$ and $(ii)$.\\
\end{proof}

Since we have the number of independent global sections  of a rank $2$ Ulrich bundle $\mathcal E$ is $12$(cf. \cite{Coskundelpezzo}, Corollary $2.5$ $(ii)$), we obtain the following corollary.\\ 

\begin{corollary}\label{c7.7}
The Brill-Noether locus $\mathcal W^{11}_2 - \mathcal W^{12}_2 \subset \mathcal M(2;5H,55)$ is non-empty and hence so is $\mathcal W^r_2$ for all $r \in \{0,...11\}$.\\
\end{corollary} 

\begin{remark}
$(i)$ Note that, Theorem \ref{T4.3} and Corollary \ref{c7.7}  completes a proof of the promised Theorem \ref{PP1}.\\

$(ii)$ In this case we see that that the generalized Brill-noether number $\rho^{11}_2 \geq  37 >0$ (which is same as the expected dimension of the moduli space $\mathcal M(2;5H,55)$) and the corresponding Brill Noether locus $\mathcal W^{11}_2$ is non-empty.\\

$(iii)$  From Remark \ref{r2.31}, it follows that any rank $2$ Ulrich bundle on $X$ is always special and therefore, from Proposition \ref{p2.36} $(b)$, it follows that for a a rank $2$ Ulrich bundle $\mathcal E$ on $X$ there exists a smooth curve $C \in \mathcal L_5$ and a line bundle $L$ with $\text{deg}(L)=55$, $H^1(L(3H)) =0$ and some $\sigma_0,\sigma_1 \in H^0(L)$ defines a base-point free pencil such that $\mathcal E$ is the dual of the kernel of the evaluation map as described in Proposition \ref{p2.36}.\\


$(iv)$ We observe that if $\mathcal E$ is a rank $2$ Ulrich bundle with $c_1(\mathcal E) = 5H$ and $c_2(\mathcal E) =55$, then for any $t \in \mathbb N$, the rank $2$ vector bundle defined by $\mathcal F_t:= \mathcal E(t)$ is an weakly Ulrich bundle with $c_1(\mathcal F_t) =(5+2t)H$ and $c_2(\mathcal F_t) = 55+30t+6t^2$. Note that, $\mathcal F_t$ is also $\mu_H$ stable for every $t$. Therefore, We have the Brill-Noether loci $\mathcal W^1_2 \subset \mathcal M(2;(5+2t)H,55+30t+6t^2)  $ are non-empty. We mention that this loci are different from the ones mentioned in Remark \ref{R4.5}, $(ii)$.\\
\end{remark}

We end this section by a discussion on the possibility of existence higher rank  Ulrich bundles on $X$. In what follows we  show the existence of  a simple Ulrich bundle of  rank $4$ under the hypothesis that there exists atleast $2$ non-isomorphic rank $2$ Ulrich bundles on $X$. The techniques are  mainly from \cite{Coskun1}.\\

Let's assume that there exists a stable Ulrich bundle $\mathcal E_1$ of rank $2r-2$ ( $r\geq 2$). Let's choose a rank $2$ Ulrich bundle on $X$ such that $ \mathcal E_1$ and $\mathcal E_2 $  are non-isomorphic for $r=2$ (Note that, this can be done using  our assumption). Then we have :
\begin{align*}
Ext^1(\mathcal E_1, \mathcal E_2) & \cong Ext^1(\mathcal E_1 \otimes \mathcal  E^*_2 , \mathcal O_X)\\
& \cong Ext^1(\mathcal E_1 \otimes \mathcal E^*_2 \otimes \mathcal O_X(2), \mathcal O_X(2))\\
& \cong H^1(\mathcal E^*_1 \otimes \mathcal E_2)
\end{align*}
This means $\text{dim}(Ext^1(\mathcal E_1, \mathcal E_2)) = h^1(\mathcal E^*_1 \otimes \mathcal E_2)$. Since, $\chi(\mathcal E_2 \otimes \mathcal E^*_1) =-26(r-1)$ (from Proposition \ref{euler}) and $h^0(\mathcal E^*_1 \otimes \mathcal E_2) =0$\footnote{Indeed we have $H^0(\mathcal E^*_1 \otimes \mathcal E_2) \cong Hom(\mathcal E_1 , \mathcal E_2)$ and  it can be shown that both $\mathcal E_1, \mathcal E_2$ are of same slope $15$. Therefore, there can be no non-trivial morphism between $\mathcal E_1$ and $\mathcal E_2$. }, we obtain $h^1(\mathcal E^*_1 \otimes \mathcal E_2)= 26(r-1) + h^2(\mathcal E^*_1 \otimes \mathcal E_2) >0$ for $r \geq 2$. Therefore, we can choose a non-split extension :
\begin{align*}
0 \to \mathcal E_2 \to \mathcal E \to \mathcal E_1 \to 0
\end{align*}
By Lemma \ref{simple}, $\mathcal E$ is simple and hence there exists a simple Ulrich bundle of rank $2r$. This means in particular, under our assumption there exists a simple rank $4$ Ulrich bundle on $X$.\\

\begin{remark}

If  a simple  Ulrich bundle $\mathcal E$ of rank $2r$ is stable, then the corresponding moduli space $\mathcal M(2r; 5rH, 75r^2 -20r)$ is non-empty. We expect that similar calculation regarding higher rank simple Ulrich bundles can be carried out for all general surfaces of even degree between $5$ and $15$.  
\end{remark}


\section*{Acknowledgement}
 I thank Dr. Sarbeswar Pal  for  many valuable comments. I thank Dr. krishanu Dan for answering a question on dimension estimate and Dr. Emre Co\c{s}kun for pointing out the relevant works regarding Ulrich bundles on surfaces. Finally, I thank Prof. Luca Chiantini for pointing out the work on the classification of ACM bundles on general sextic surfaces.


\begin{thebibliography}{1}
\bibitem{AF} Aprodu, M., Farkas, G.:
Green's conjecture for curves on arbitrary K3 surfaces. 
Compos. Math. 147,  3, (2011)  839–851 
\bibitem{ACGH}Arbarello, E., Cornalba, M., Griffiths, P. A., Harris, J.:
Geometry of algebraic curves. Vol. I.
Grundlehren der Mathematischen Wissenschaften [Fundamental Principles of Mathematical Sciences], 267. Springer-Verlag, New York, (1985)
\bibitem{Basili} Basili, B.: Indice de Clifford des intersections complètes de l'espace. (French) [Clifford index of complete intersections in space] Bull. Soc. Math. France 124,1, (1996)  61–95
\bibitem{BIU} Beauville, A.:
An introduction to Ulrich bundles.
Eur. J. Math. 4, 1, (2018)   26–36 
\bibitem{Beauvillehypersurface}Beauville, A.:
Determinantal hypersurfaces.
Dedicated to William Fulton on the occasion of his 60th birthday.
Michigan Math. J. 48, (2000) 39–64
\bibitem{DS}Bhattacharya, D., Pal, S.:Geometry of some moduli of bundles over a very general sextic surface for small second Chern classes and Mestrano-Simpson Conjecture. arXiv preprint arXiv:2003.06146 (2020).
\bibitem{HartshorneUlrich} Casanellas, M., Hartshorne, R., Geiss, F; Schreyer, F.-O.:Stable Ulrich bundles. Internat. J. Math. 23,  8, 1250083 (2012)
\bibitem{Casnati} Casnati, G.: Examples of smooth surfaces in $\mathbb P^3$ which are Ulrich-wild. Bull. Korean Math. Soc. 54, 2,  (2017)  667–677
\bibitem{Chiantini5} Chiantini, L.,  Faenzi, D.:
Rank 2 arithmetically Cohen-Macaulay bundles on a general quintic surface.
Math. Nachr. 282,  12,  (2009)  1691–1708
\bibitem{Coskunsurvey}Coskun, E.:
A survey of Ulrich bundles.  Analytic and algebraic geometry, 85–106, Hindustan Book Agency, New Delhi (2017)
\bibitem{Coskun1} Coskun, E.:
Ulrich bundles on quartic surfaces with Picard number 1.
C. R. Math. Acad. Sci. Paris 351, 5-6,  (2013)  221–224
\bibitem{Coskunternary} Coskun, E., Kulkarni, R. S., Mustopa, Y.: On representations of Clifford algebras of ternary cubic forms. New trends in noncommutative algebra,  Contemp. Math.,  Amer. Math. Soc., Providence, RI, 562,  (2012) 91-99
\bibitem{Coskunconstruction}Coskun, E.,  Kulkarni, R. S., Mustopa, Y.:
Pfaffian quartic surfaces and representations of Clifford algebras.
Doc. Math. 17,  (2012) 1003–1028
\bibitem{Coskundelpezzo} Coskun, E., Kulkarni, R. S., Mustopa, Y.: The geometry of Ulrich bundles on del Pezzo surfaces. J. Algebra 375,  (2013) 280–301
\bibitem{LC1} Costa, L.,  Miró-Roig, R. M.: 
Brill-Noether theory for moduli spaces of sheaves on algebraic varieties. 
Forum Math. 22,  3,  (2010)  411–432
\bibitem{LC2} Costa, L., ; Miró-Roig, R. M.:
Brill-Noether theory on Hirzebruch surfaces.
J. Pure Appl. Algebra 214,  9,  (2010) 1612–1622
\bibitem{KS} Dan, K., Pal, S.: Non-emptiness of Brill-Noether loci over very general quintic hypersurface. Bull. Sci. Math. 147 , 83–91 (2018)
\bibitem{Eisenbud} Eisenbud, D., Schreyer, F.-O., Weyman, J.: Resultants and Chow forms via exterior syzygies. J. Amer. Math. Soc. 16,  3,  (2003)  537–579
\bibitem{LG0} Göttsche, L.: Hilbert schemes of zero-dimensional subschemes of smooth varieties. Lecture Notes in Mathematics, 1572. Springer-Verlag, Berlin, (1994)
\bibitem{LG1} Göttsche, L.,  Hirschowitz, A.:
Weak Brill-Noether for vector bundles on the projective plane. Algebraic geometry (Catania, 1993/Barcelona, 1994),
Lecture Notes in Pure and Appl. Math.,  Dekker, New York, 200,  (1998) 63-74
\bibitem{GH} Griffiths, P., Harris, J.: On the Noether-Lefschetz theorem and some remarks on codimension-two cycles. Math. Ann. 271 ,  1,  (1985)  31–51
\bibitem{alternativepetri}  Grzegorczyk, I., Teixidor i Bigas, M.:  Brill-Noether theory for stable vector bundles. Moduli spaces and vector bundles,  London Math. Soc. Lecture Note Ser., 359, Cambridge Univ. Press, Cambridge,  (2009)  29–50
\bibitem{HarrisAG}  Harris, Joe.:  Algebraic geometry. A first course. Corrected reprint of the 1992 original. Graduate Texts in Mathematics. Springer-Verlag, New York, 133 (1995)
\bibitem{RH}  Hartshorne, R.:  Algebraic geometry. Graduate Texts in Mathematics. Springer-Verlag, New York-Heidelberg, 52 (1977)
\bibitem{HeMin}  He, M.: Espaces de modules de systèmes cohérents. (French) [Moduli spaces of coherent systems] Internat. J. Math. 9,  5,  (1998)  545–598
\bibitem{Huy}  Huybrechts, D., Lehn, M.: The geometry of moduli spaces of sheaves. Second edition. Cambridge Mathematical Library. Cambridge University Press, Cambridge (2010)
\bibitem{KC} Keem, C.:
Double coverings of smooth algebraic curves. Algebraic geometry in East Asia (Kyoto, 2001), World Sci. Publ., River Edge, NJ,  (2002) 75–111
\bibitem{LP}  Le Potier, J.: Systèmes cohérents et structures de niveau. (French) [Coherent systems and level structures] Astérisque. 214 , 143  (1993)
\bibitem{Leyenson}  Leyenson, M.: On the Brill-Noether theory for K3 surfaces. Cent. Eur. J. Math. 10,  4,  (2012)  1486–1540
\bibitem{Markman} Markman, E.: 
Brill-Noether duality for moduli spaces of sheaves on K3 surfaces.
J. Algebraic Geom. 10,  4,  (2001)  623–694
\bibitem{Simpson11} Mestrano, N., Simpson, C.: 
Obstructed bundles of rank two on a quintic surface.
Internat. J. Math. 22,  6,  (2011) 789–836
\bibitem{Simpson18}  Mestrano, N., Simpson, C.: Irreducibility of the moduli space of stable vector bundles of rank two and odd degree on a very general quintic surface. Pacific J. Math. 293,  1,  (2018) 121–172
\bibitem{Nakashima} Nakashima, T.:
Brill-Noether problems in higher dimensions. 
Forum Math. 20, 1,  (2008) 145–161
\bibitem{S} Pal, S.: Irreducibility of moduli of vector bundles over a very general sextic Surface. arXiv preprint arXiv:2003.07036 (2020)
\bibitem{ACM6} Patnott, M.:
The h-vectors of arithmetically Gorenstein sets of points on a general sextic surface in $\mathbb P^3$. 
J. Algebra 403,  (2014)  345–362
\bibitem{Y1} Yoshioka, K.:
Some examples of Mukai's reflections on K3 surfaces.
J. Reine Angew. Math. 515,  (1999) 97–123
\bibitem{Y2} Yoshioka, K.: Brill-Noether problem for sheaves on K3 surfaces. (Proceedings of the Workshop Algebraic Geometry and Integrable Systems related to String Theory),  (2001) 109-124


		
		
\end{thebibliography}
\end{document}